\documentclass[11pt]{article}

\usepackage{amsmath,amsthm}

\usepackage{amssymb,latexsym}

\usepackage{enumerate}

\usepackage[mathscr]{euscript}
\usepackage{amsmath}

\usepackage[colorlinks=true,linkcolor=blue,citecolor=magenta]{hyperref}
\usepackage{todonotes}

 \newcommand{\BB}{{\mathcal  B}}

\newcommand{\EE}{{\mathcal  E}}
\newcommand{\FF}{{\mathcal  F}}
\newcommand{\GG}{{\mathcal  G}}

\newcommand{\TT}{{\mathcal  T}}

\newcommand{\RR}{{\mathcal  R}}

\newcommand{\VV}{{\mathcal  V}}

\newcommand{\BC}{{\mathbb C}}

\newcommand{\BE}{{\mathbb E}}
\newcommand{\BF}{{\mathbb F}}
\newcommand{\BG}{{\mathbb G}}

\newcommand{\BN}{{\mathbb N}}
\newcommand{\BP}{{\mathbb P}}

\newcommand{\BR}{{\mathbb R}}

\newcommand{\fch}{{\mathbf{1}}}

\newcommand{\arrowd}{\mathop{\rightarrow}_d}
\newcommand{\arrowp}{\mathop{\rightarrow}_P}
\newcommand{\arrowdl}{\mathop{\longrightarrow}_d}
\newcommand{\arrowpl}{\mathop{\longrightarrow}_P}
\newcommand{\arroww}{\mathop{\rightarrow}_{w}}
\newcommand{\arrowwl}{\mathop{\longrightarrow}_{w}}

\newtheorem{theorem}{\bf Theorem}[section]
\newtheorem{proposition}[theorem]{\bf Proposition}
\newtheorem{lemma}[theorem]{\bf Lemma}
\newtheorem{corollary}[theorem]{\bf Corollary}

\theoremstyle{definition}
\newtheorem{definition}[theorem]{Definition}
\newtheorem{example}[theorem]{\bf Example}

\newtheorem{remark}[theorem]{Remark}

\numberwithin{equation}{section}

\begin{document}

\title{Stability of Solutions to Semilinear Exterior Dirichlet Problems under Convergence of L\'evy Symbols}

\author{Tomasz Klimsiak\footnote{E:mail address: {\tt tomas@ mat.umk.pl} (corresponding author)}\,,\,\,\, Andrzej Rozkosz, Leszek S\l omi\'nski\\
{\small Faculty of Mathematics and Computer Science,
	Nicolaus Copernicus University} \\
{\small  Chopina 12/18, 87--100 Toru\'n, Poland} }
\date{}
\maketitle
\begin{abstract}
We study stability of semilinear exterior Dirichlet problems driven by possibly nonsymmetric L\'evy operators. 
Let $L^n$ and $L$ have L\'evy–Khintchine symbols $\psi_n$ and $\psi$, respectively. Our basic assumption on the operators is that 
$\psi_n(\xi)\to\psi(\xi)$ for every $\xi\in\mathbb R^d$. Under mild regularity assumptions on the domain and an absolute-continuity 
condition for the resolvent of the limiting process, we prove pointwise and $L^p$-stability of solutions. 
No symmetry, uniform ellipticity, comparability of L\'evy kernels, or common energy space is required; nonzero exterior 
data are allowed; and the monotone nonlinearity is subject to no growth restriction in the solution variable. The framework 
therefore covers local, nonlocal and mixed operators and, in particular, nonlocal-to-local limits, as well as fractional, relativistic, 
anisotropic and nonsymmetric stable-type operators. The proofs combine convergence of L\'evy processes with a detailed analysis 
of exit times and exit positions and a stability theorem for monotone backward stochastic differential equations with varying terminal times. 
In the symmetric case, we also obtain Mosco convergence of the corresponding part forms on $C^0$ domains and give a 
counterexample showing that whole-space Mosco convergence need not pass to part forms on arbitrary open sets.
\end{abstract}

\section{Introduction}
\label{sec1}

Let $D\subset\BR^d$, $d\geq1$, be a bounded open set. We study 
stability, with respect to the operator, of the semilinear exterior Dirichlet
problem
\begin{equation}
\label{eq1.1}
\begin{cases}
-Lu=f(\cdot,u) & \text{in }D,\\
u=g & \text{in }D^c:=\BR^d\setminus D,
\end{cases}
\end{equation}
where $f:\BR^d\times\BR\rightarrow\BR$,  
$g:\mathbb R^d\rightarrow\BR$ are measurable functions and $L$ is a (possibly nonsymmetric) L\'evy operator. On
$C_c^2(\mathbb R^d)$ such an operator has the form
\begin{align}
\label{eq1.2}
L\varphi(x)
={}&b\cdot\nabla\varphi(x)+\operatorname{Tr}(Q\nabla^2\varphi(x))\nonumber\\
&+\int_{\mathbb R^d}\big(\varphi(x+y)-\varphi(x)
-y\cdot\nabla\varphi(x)\mathbf 1_{\{|y|\leq1\}}\big)\,\nu(dy),
\end{align}
where $b\in\mathbb R^d$, $Q$ is a symmetric, nonnegative definite \(d\times d\) matrix and
$\nu$ is a L\'evy measure, i.e. a positive Borel measure on $\BR^d$ such that
$\nu(\{0\})=0$ and
\[
\int_{\BR^d}(1\wedge |y|^2)\,\nu(dy)<\infty.
\]
Equivalently, the Fourier transform of $L\varphi$ has the form
\begin{equation}
\label{eq1.8}
\widehat{L\varphi}(\xi)=-\psi(\xi)\widehat\varphi(\xi),
\end{equation}
where $\psi$ is the L\'evy--Khintchine symbol of $L$, i.e.
\begin{equation}
\label{eq1.4}
\psi(\xi)=-i(b,\xi)+(\xi,Q\xi)
-\int_{\BR^d}(e^{i(\xi,y)}-1-i(\xi,y)\fch_{\{|y|\le 1\}})
\,\nu(dy),\quad\xi\in\BR^d.
\end{equation}

For $n\geq1$, let $L^n$ be a L\'evy operator with symbol $\psi_n$ and
$u_n$ be a solution of \eqref{eq1.1} with $L$ replaced by $L^n$. Our basic
assumption  on the operators $L^n$ and $L$ is
\begin{enumerate}
\item[(H1)]
$\lim_{n\rightarrow\infty}\psi_n(\xi)=\psi(\xi)$ for every $\xi\in\mathbb R^d$.
\end{enumerate}
The simplest nontrivial case is 
\begin{equation}
\label{eq1.5}
L^n=-(-\Delta)^{\alpha_n/2},\quad L=-(-\Delta)^{\alpha/2}\quad \mbox{with}\quad (0,2)\ni\alpha_n\rightarrow\alpha\in(0,2].
\end{equation}
Then clearly (H1) is satisfied since in this case
\[
\psi_n(\xi)=|\xi|^{\alpha_n},\quad \psi(\xi)=|\xi|^{\alpha},\quad\xi\in\BR^d.
\]
The main question is whether for given data $f,g$  assumption (H1) implies that
\begin{equation}
\label{eq1.6}
u_n(x)\rightarrow u(x)\quad\text{for every }x\in D
\end{equation}
or
\begin{equation}
\label{eq1.7}
u_n\rightarrow u\quad\text{in }L^p(D)
\end{equation}
for some $p\ge1$. 

The point is that (H1) is an intrinsic and very weak mode of convergence. It does not require a common integral representation 
of the operators, convergence or comparability of their Lévy kernels, symmetry, uniform ellipticity, or a common energy space. 
In particular, the diffusion, drift, and jump parts may all vary; a diffusion component may arise as a limit of small jumps, while either the 
diffusion or the jump component may disappear in the limit. Thus local, purely nonlocal, and mixed integro-differential operators may 
occur within the same approximation scheme. This degree of generality is the main source of difficulty in studying stability of solutions of problems of the form \eqref{eq1.1}.

The model family \eqref{eq1.5} already illustrates two basic aspects of the problem. If \(\alpha_n\to\alpha\in(0,2)\), the question concerns stability 
of solutions with respect to the order of the fractional Laplacian, whereas the case \(\alpha_n\to2\) includes the nonlocal-to-local transition. 
More generally, (H1) allows the approximating and limiting operators to have qualitatively different local and nonlocal structures. 
This motivates studying \eqref{eq1.6} and \eqref{eq1.7} directly under convergence of symbols, rather than within a fixed structural class of operators.

The generality of the problem is not confined to the operators. For the monotone nonlinearities considered below, 
we impose no growth restriction on \(f(x,y)\) as \(|y|\to\infty\). Thus our stability results combine a weak mode of convergence 
of the generators with arbitrary monotone growth of the semilinear term and allow nonzero exterior data.

For some specific families of nonlocal operators, stability has usually been
proved by analytic methods adapted to the family under consideration. In the
fractional-to-local regime one can use explicit information about the kernel,
uniform fractional Sobolev estimates, and comparison of the nonlocal and local
energies; see, for example, \cite{BH-S,BdT,JSW}. Variational methods are also
effective for symmetric operators acting on compatible energy spaces, as in
\cite{FGKV}, and recent nonlocal-to-local results for divergence-form
operators obtained in \cite{ADLSV} rely on estimates tailored to that structure.
Compactness and half-relaxed-limit arguments can be used in the viscosity
setting once consistency of the operators, local compactness of the solutions,
and stability of the boundary condition have been established; see, e.g.,
\cite{BarlesImbert2008}. Semilinear fractional problems have also been treated
through the special variational structure of their ground states; see
\cite{BS1,BS2}.

None of these mechanisms follows from (H1) in the present generality. A
variational argument cannot cover arbitrary nonsymmetric L\'evy operators and,
even in the symmetric case, convergence of the whole-space forms does not
automatically pass to their parts on $D$. Pointwise convergence of 
symbols supplies neither uniform regularity of $u_n$ nor estimates comparing
the Green functions of the operators $L^n$. Moreover, the exterior boundary 
condition is itself
operator-dependent: different jump measures transmit the values of $g$ on
$D^c$ into $D$ in different ways. In probabilistic terms, both the first exit
time from $D$ and, when $g\neq0$, the exit position vary with the operator.
Thus neither an energy compactness argument, a direct passage to the limit in
the equation, nor a general viscosity stability theorem is available solely
under (H1).

\subsection{Main results}

The natural formulation of \eqref{eq1.1} in this generality is the
integral equation
\begin{equation}
\label{eq1.11}
u(x)=P_Dg(x)+G_D[f(\cdot,u)](x),\qquad x\in D,
\end{equation}
where $P_D$ and $G_D$ are, respectively, the Poisson and Green operators
associated with $L$ and $D$. This formulation treats local, nonlocal, and
mixed operators in the same way and does not presuppose symmetry or Sobolev
regularity of a solution.

As for the data $f,g$, we assume that 
\begin{enumerate}
\item[(H2)]$f\in C(\bar D\times\BR)$ and $g\in C(\BR^d)$;
	
\item[(H3)] $\big(f(x,y_1)-f(x,y_2)\big)(y_1-y_2)\leq0$ for any  $x\in D$ and $y_1,y_2\in\mathbb R$.
\end{enumerate}
The generality of our stability results lies both in the class of operators and in the nonlinearities.
Alongside the pointwise convergence assumption
{\rm(H1)} on the symbols, we impose no restriction on the growth
of $f(x,y)$ as $|y|\to\infty$, beyond the monotonicity required
by {\rm(H3)}. In particular, neither a Lipschitz condition nor
a polynomial or subcritical growth bound is needed.

The general pointwise stability result also allows unbounded exterior data
under suitable pointwise integrability and uniform-integrability assumptions.
For bounded exterior data, however, the result takes a particularly simple
form. Suppose that $D$ is of class $C^0$,
that the resolvent kernel of the limiting operator is absolutely continuous (i.e. the so-called condition (ACP) is satisfied)
and that $g\in C_b(\mathbb R^d)$. Then (H1)--(H3) imply (\ref{eq1.6}). Notice that no differentiability of $\partial D$ is required in
\eqref{eq1.6}.

Condition (ACP) holds, in particular, if the limiting symbol satisfies
the Hartman--Wintner condition
\begin{equation}
\label{eq1.9}
\lim_{|\xi|\to\infty}
\frac{\operatorname{Re}\psi(\xi)}{\log(1+|\xi|)}=\infty.
\end{equation}

The $L^p$ version of the stability result holds under weaker assumptions on $f$. Let $1\leq p<\infty$  and assume that
\begin{equation}
\label{eq1.10}
f(\cdot,y)\in L^1(D)\mbox{ for every }y\in\BR,\quad f(\cdot,0)\in L^p(D),
\quad f-f(\cdot,0)\in C(\bar D\times\BR).
\end{equation}
Then, under {\rm(H1)}, {\rm(H3)}, {\rm(ACP)}, and the same
conditions on $D$ and $g$ as above, \eqref{eq1.7} holds.
The convergence can be strengthened to pointwise convergence when the Green kernels of the limiting and approximating 
operators are uniformly bounded in $L^q$  for some $q>1$ and $f$ satisfies (\ref{eq1.10}) with $p$ being H\"older conjugate of $q$.

The results described  above are formulated for solutions of the integral equation
\eqref{eq1.11}, but they also apply to standard analytic notions of solution.
If the sector condition
\begin{equation}
\label{eq1.3}
|\operatorname{Im}\psi|\leq c\operatorname{Re}\psi
\end{equation}
holds, the operator $L$ is associated with a (possibly nonsymmetric) Dirichlet form.
We show  that under the transience and integrability assumptions specified in
Section~\ref{sec3}, every weak Sobolev solution has a version satisfying
\eqref{eq1.11}. We also prove that bounded continuous viscosity solutions
with bounded continuous right-hand side satisfy (\ref{eq1.11}).
Consequently, under the additional conditions stated above, \eqref{eq1.6} and \eqref{eq1.7} are also 
stability statements for weak or viscosity solutions. The use of the integral formulation is therefore a way of unifying analytic solution concepts rather than replacing them by a different problem.

Our main stability result applies to wide classes of operators, for instance fractional Laplace operators (\ref{eq1.5}),  multifractal operators,
relativistic stable operators, sums of local and nonlocal operators, logarithmic operators of Schr\"odinger type, anisotropic stable operators with varying spectral
measures or nonsymmetric stable-like operators. The approximating and
limiting operators may thus belong to qualitatively different classes. To the best of our knowledge, no previous result simultaneously covers this range 
of operators and provides pointwise or $L^p$ convergence for semilinear Dirichlet problems with nonzero exterior data and monotone nonlinearities without growth restrictions.

There is a further, entirely variational consequence. When $\psi_n$ and
$\psi$ are real-valued, (H1) implies Mosco convergence of the corresponding
whole-space Dirichlet forms (see \cite{CS}). We prove that under (ACP)  the parts of these forms on $C^0$ domains
also converge in the sense of Mosco. Such a  convergence result is
stated in \cite[Theorem 3.3]{CS} for arbitrary open sets
$D\subset\BR^d$. We give a counterexample showing that this assertion
does not hold in that generality: the whole-space forms may converge while
their parts fail the Mosco lower bound. Thus some regularity of $D$  is
indispensable; taking the part on $D$ is not a continuous operation for an
arbitrary open set. 

\subsection{Idea of the proof}

In our proofs of the main results of the paper we use a probabilistic approach. It allows us to overcome the compactness 
and boundary difficulties that arise under the weak convergence assumption (H1).
Let $X^{x}$ (resp. $X^{x,n}$) be a L\'evy process starting from $x$ with the L\'evy--Khintchine exponent 
$\psi$ (resp. $\psi_n$). Denote by $\tau^x_D$ the first exit time of $X^x$ from $D$, i.e. 
\[
\tau^x_D=\inf\{t>0:X^x_t\notin D\},
\]
and by $\tau^{x,n}_D$ denote the first exit time of $X^{x,n}$ from $D$. Then  \eqref{eq1.11} can be written in the form
\begin{equation}
\label{eq1.12}
u(x)=\BE g(X^x_{\tau^x_D})+\BE\int_0^{\tau^x_D}f(X^x_t,u(X^x_t))\,dt
\end{equation}
($\BE$ denotes the expectation), and the  solution $u_n$ of (\ref{eq1.1}) 
with \(L\) replaced by \(L^n\) can be written as
\begin{equation}
\label{eq1.13}
	u_n(x)=\BE g(X^{x,n}_{\tau_D^{x,n}})+\BE\int_0^{\tau^{x,n}_D}f(X^{x,n}_t,u_n(X^{x,n}_t))\,dt,
\end{equation}
where $\tau^{x,n}_D$ is the first exit time of $X^{x,n}$ from $D$. It is known that (H1) implies the convergence of $X^{x,n}$ to $X^{x}$ in the Skorokhod space, but not automatically of their exit times $\tau^{x,n}_D$ or exit positions $X^{x,n}_{\tau^{x,n}_D}$ to the corresponding quantities for  $X^x$.
These are precisely the quantities that enter \eqref{eq1.12}, \eqref{eq1.13}. We identify
the path mechanisms responsible for the possible failure of such convergences
and formulate conditions, in terms of the limiting process and the geometry
of $\partial D$, which exclude them. If $D$ is of class $C^0$, then these conditions are satisfied for any $\psi$ for which (ACP) is satisfied. The proof of this fact is based on some structural properties of L\'evy processes.

Once the  convergence of the relevant exit quantities is established, a new
stability theorem for monotone backward stochastic differential equations
with varying terminal times permits us to  pass from (\ref{eq1.13}) to  \eqref{eq1.12}.
In analytic terms, this yields the convergence of solutions of the nonlinear 
equations of the form \eqref{eq1.11}, but involving Poisson and Green kernels for $L^n$, to the solution of \eqref{eq1.11} without requiring any estimates for the Green or Poisson kernels themselves.

The idea of proving stability results by using convergence results for processes and  backward stochastic 
differential equations has appeared in the recent paper \cite{RS2} in  which semilinear evolution equations on the whole space are treated. 
The new issue in the present paper is the exterior Dirichlet condition: the first exit time and,
when $g\neq0$, the exit position both vary with the operator. The problem of convergence of these quantities is rather subtle. Related
convergence questions for reflected processes and linear Neumann problems are
studied by A.~Rozkosz and L.~S\l omi\'nski in the preprint \emph{Reflected Skorokhod equations and the Neumann boundary value problem for elliptic equations with L\'evy-type operators} (arXiv:2303.12664, 2023), where the boundary mechanism is different.

Although the results are stated for monotone nonlinearities,
the probabilistic approach is not confined to absorption problems.
The representation of solutions and the analysis of exit times
and exit positions provide a framework for treating other
nonlinearities, provided that suitable well-posedness and stability
estimates can be established for the associated BSDEs.

\section{Stability of solutions of BSDEs}
\label{sec2}

We denote by $D([0,T];\BR^d)$ the space of $\BR^d$-valued c\`adl\`ag functions on $[0,T]$ (right continuous with left-hand limits). 
We equip it with the Skorokhod $J_1$ topology. For a definition of this topology and its properties we refer to \cite{JS}. For a process  
$X$ with paths in $D([0,T];\BR)$ we set  $\Delta  X_t=X_t-X _{t-}$ for $t\in[0,T]$,
with the convention $X_{0-}=0$.

\subsection{BSDEs with random terminal time}

In this section $(\Omega_0,\mathcal G,\mathbb G=(\mathcal G_t)_{t\ge0},\BP)$ is a complete 
filtered probability space satisfying the usual hypotheses (see, e.g., \cite[Section I.1]{P}). 

Recall that an  $\BG$-adapted process $Y$ is of
Doob class D  if the collection $\{Y_{\tau},\tau\in\TT\}$,
where $\TT$ is the set of all almost surely finite   $\mathbb G$-stopping times,
is uniformly integrable.

\begin{definition}
Let  $\tau\in\TT$, $\zeta$ be a $\GG_{\tau}$-measurable  random variable,
$V$ be a finite-variation $\BG$-adapted c\`adl\`ag process with $V_0=0$,
and $X$ be a $\BG$-adapted c\`adl\`ag process such that $\mathbb P(X_t\in D,\, t\in [0,\tau))=1$.
We say that  a pair $(Y,M)$ of $\mathbb G$-adapted c\`adl\`ag processes is a solution of the backward
stochastic differential equation with terminal condition $\zeta$  at
terminal time $\tau$  and coefficient $f+dV$
(BSDE$^\tau(\zeta,f+dV,X,\BG,\BP)$ for short) if
\begin{enumerate}
\item[(a)]$Y$ is of class D  and  $M$ is a local martingale with $M_0=0$,
\item[(b)]  $\BP\left( \int_0^\tau|f(X_t,Y_t)|\,dt<\infty\right)=1$ 
and
\[
Y_t=\zeta+\int^{\tau}_{t}f(X_s,Y_s)\,ds+(V_\tau-V_t)
-\int^{\tau}_{t}dM_s,\quad t\in[0,\tau],\quad
\BP\mbox{-a.s.}
\]
\end{enumerate}
\end{definition}

\begin{remark}
It is known (see, e.g., \cite{KR:NoD}) that if a pair $(Y,M)$ is a solution of 
 BSDE$^\tau(\zeta,f+dV,X,\BG,\mathbb P)$ such that  
\begin{equation}
\label{eq2.18}
\BE\int^{\tau}_0|f(X_s,Y_s)|\,ds<\infty,
\end{equation} 
and $\mathbb E|V|_\tau<\infty$,   then $M$ is a uniformly integrable martingale. We write \(|V|_t\) for the total variation of \(V\) on \([0,t]\).
\end{remark}

\begin{theorem}
\label{th2.2}
Let  $\tau\in\mathcal T$, $\zeta$ be an $\GG_{\tau}$-measurable  random variable,
$V$ be a finite-variation $\mathbb G$-adapted c\`adl\`ag process with $V_0=0$
and $f:D\times\BR\to\BR$ be a measurable function 
satisfying \mbox{\rm(H3)}. Then 
\begin{enumerate}[\rm(i)]
\item There exists at most one solution to $\operatorname{BSDE}^\tau(\zeta,f+dV,X,\BG,\BP)$.

\item If additionally  $f(x,\cdot)$ is continuous for every $x\in D$, $\BE\int^{\tau}_0|f(X_t,y)|\,dt<\infty$
for every $y\in\BR$ and $\BE|\zeta|+\mathbb E|V|_\tau<\infty$, then there exists a  solution $(Y,M)$ 
of $\operatorname{BSDE}^\tau(\zeta,f+dV,X,\BG,\BP)$. 
Moreover, $M$ is a uniformly integrable martingale and
\[
\BE\int^{\tau}_0|f(X_t,Y_t)|\,dt\le \mathbb E|\zeta|+2\BE\int^{\tau}_0|f(X_t,0)|\,dt
+\BE|V|_\tau.
\]
\end{enumerate}
\end{theorem}
\begin{proof}
The desired result follows from
\cite[Propositions 2.4 and 2.8]{K:SPA} with $S=0$ (or   
\cite[Theorem 3.7]{KR:CM} in the case $\zeta=0$). 
\end{proof}

Let $\zeta=g(X_{\tau})$. Of course, if $\BE(|\zeta|+\tau)<\infty$ and $f,g$ satisfy (H2) and (H3), 
then $\zeta,f$ satisfy the assumptions of the above theorem.  

Let us recall that
a function  $F:[0,\infty)\to[0,\infty)$ with $F(0)=0$ is called moderate if there exists $c>0$ such that
$F(2x)\le cF(x),\, x>0$.

\begin{proposition}
\label{prop2.4}
Let  $\tau\in\TT$, $\zeta,\zeta'$ be $\GG_\tau$-measurable random variables such that
$\BE|\zeta|+\BE|\zeta'|<\infty$, $V,V'$ be  finite-variation $\mathbb G$-adapted c\`adl\`ag processes with $V_0=V'_0=0$, and $f,f':D\times\BR\to\BR$ 
be  measurable functions such that $f(x,\cdot),f'(x,\cdot)$ are continuous for every $x\in D$ and  $f$ satisfies \mbox{\rm(H3)}. 
Let $(Y,M)$ be a solution to $\operatorname{BSDE}^\tau(\zeta,f+dV,X,\mathbb G,\BP)$.  
Then
\begin{enumerate}[\rm(i)]
\item For every moderate convex function
$F:[0,\infty)\to[0,\infty)$ there exists a
constant $C_F>0$ depending only on $F$ such that for any solution $(Y',M')$ of  
$\operatorname{BSDE}^\tau(\zeta',f'+dV',X,\mathbb G,\BP)$,
\begin{align*}
&\BE\Big[
F\Big(\int_0^\tau |f(X_s,Y_s)-f'(X_s,Y'_s)|\,ds\Big)\Big]
\\&\quad\quad\quad\leq
C_F\,
\BE\Big[F\Big(
|\zeta-\zeta'|+\int_0^\tau |f(X_s,Y'_s)-f'(X_s,Y'_s)|\,ds+|V-V'|_\tau\Big)\Big].
\end{align*}

\item For any solution $(Y',M')$ of  $\operatorname{BSDE}^\tau(\zeta',f'+dV',X,\BG,\BP)$ and any $q\in (0,1)$ one has
\[
\BE\sup_{t\le\tau}|Y_t-Y'_t|^q\le c_q \Big( \mathbb E\Big[|\zeta-\zeta'|+\int_0^\tau |f(X_s,Y'_s)-f'(X_s,Y'_s)|\,ds+|V-V'|_\tau\Big]\Big)^q.
\]

\item For any solution $(Y',M')$ of  $\operatorname{BSDE}^\tau(\zeta',f'+dV',X,\mathbb G,\BP)$ and any $p>1$ one has
\[
\BE\sup_{t\le\tau}|Y_t-Y'_t|^p\le c_p  \mathbb E\Big(|\zeta-\zeta'|+\int_0^\tau |f(X_s,Y'_s)-f'(X_s,Y'_s)|\,ds+|V-V'|_\tau\Big)^p.
\]
\end{enumerate}
\end{proposition}
\begin{proof}
We assume that $\BE(|\zeta-\zeta'|+\int_0^\tau |f(X_s,Y'_s)-f'(X_s,Y'_s)|\,ds+|V-V'|_\tau)<\infty$, 
otherwise the asserted inequalities trivially 
hold. By (H3) and the Tanaka--Meyer  formula (see, e.g., \cite[p. 216]{P}) applied  to $|Y-Y'|$  we get, after a standard localization argument, that for any $\BG$-stopping time $\sigma$,
\begin{align*}
&|Y_{\sigma\wedge\tau}-Y'_{\sigma\wedge\tau}|\\
&\qquad\le \BE\big( |\zeta-\zeta'|+\int_{\sigma\wedge\tau}^\tau 
\operatorname{sgn}(Y_s-Y_s')(f(X_s,Y_s)-f'(X_s,Y'_s))\,ds
+|V-V'|_\tau\big|\,\GG_\sigma\Big)\\
&\qquad=
\BE\Big(|\zeta-\zeta'|-\int_{\sigma\wedge\tau}^\tau \left|f(X_s,Y_s)-f'(X_s,Y'_s)
+[f'(X_s,Y'_s)-f(X_s,Y'_s)]\right|\,ds\big|\,\GG_\sigma\Big)\\
&\qquad\quad-\BE\Big(\int_{\sigma\wedge\tau}^\tau  \operatorname{sgn}(Y_s-Y_s')[f'(X_s,Y'_s)-f(X_s,Y'_s)]\,ds-|V-V'|_\tau\,\big|\, \GG_\sigma\Big)\\
&\qquad\le 
\BE\Big( |\zeta-\zeta'|-\int_{\sigma\wedge\tau}^\tau \left|f(X_s,Y_s)-f'(X_s,Y'_s)\right|\,ds\,\big|\,\GG_\sigma\Big)\\
&\quad\qquad+2\BE \Big(\int_{0}^\tau |f'(X_s,Y'_s)-f(X_s,Y'_s)|\,ds+|V-V'|_\tau\,\big|\, \GG_\sigma\Big).
\end{align*}
Hence
\begin{align}
\label{eq2.1}
&|Y_{\sigma\wedge\tau}-Y'_{\sigma\wedge\tau}|
+\BE\Big(\int_{\sigma\wedge\tau}^\tau \left|f(X_s,Y_s)-f'(X_s,Y'_s)\right|\,ds\,\big|\, \GG_\sigma\Big)\nonumber\\
 &\qquad\le \mathbb E\Big(|\zeta-\zeta'|+2\int_{0}^\tau |f'(X_s,Y'_s)-f(X_s,Y'_s)|\,ds+2|V-V'|_\tau\,\big|\,\GG_\sigma\Big).
\end{align}
Define
\[
A_t=\int_0^{t\wedge\tau}|f(X_s,Y_s)-f'(X_s,Y'_s)|\,ds,
\qquad t\ge 0,
\]
and
\[
U=|\zeta-\zeta'|+2\int_0^\tau |f'(X_s,Y'_s)-f(X_s,Y'_s)|\,ds+2|V-V'|_\tau.
\]
Then $A$ is continuous, adapted and increasing, $A_0=0$, and
\eqref{eq2.1} implies that
\[
\BE\left[A_\infty-A_\sigma\right]\leq\mathbb E(\fch_{\{\sigma<\infty\}}U).
\]
The Garsia--Neveu lemma
\cite[Lemma 1.2(b), p.~31]{LLP}
therefore gives (i).  Furthermore, 
\eqref{eq2.1}  when combined with  \cite[Lemma 6.1]{BDHPS} yields (ii), and
when combined with  Doob's $L^p$-inequality  yields (iii).
\end{proof}

\begin{remark}
The conclusion of \cite[Theorem 2.6]{RS2} remains valid
if assumption {\rm(H5)} therein is replaced by
\[
\tag{A}\int_0^T F_r(t)\,dt<\infty,\qquad
F_r(t):=\sup_{|x|+|y|\le r}
|f(t,x,y)-f(t,x,0)|,\qquad r>0.
\]
Indeed, define $f_m$ and $Y^n, Y^{n,(m)}$ as in Step~2  of the proof of \cite[Theorem 2.6]{RS2} and then set
\[
\begin{aligned}
a_s^{n,m}
&=f(s,X_s^n,Y_s^n)
   -f_m(s,X_s^n,Y_s^{n,(m)}),\\
b_s^{n,m}
&=f(s,X_s^n,Y_s^{n,(m)})
   -f_m(s,X_s^n,Y_s^{n,(m)})\ge 0
\end{aligned}
\]
(the last inequality follows from $f_m\le f$; see \cite[(2.12)]{RS2}).
Since $Y^n\ge Y^{n,(m)}$,  monotonicity gives
$(a_s^{n,m})^+\le b_s^{n,m}$.
The local domination {\rm(A)}, convergence of $(X^n,\mathbb F^n)$,  the boundedness
in \cite[(2.14)]{RS2}, and the locally uniform approximation
$f_m\to f$ imply that for every $\varepsilon>0$,
\[
\lim_{m\to\infty}\limsup_{n\to\infty}
\BP\Big(\int_0^T b_s^{n,m}\,ds>\varepsilon\Big)=0.
\]
Moreover, the bounded-data assumptions of Step~2 and
Proposition \ref{prop2.4}(i), applied with $F(u)=u^2$
to each solution and the zero solution, yield
\[
\sup_{n,m}\mathbb E\Big(\int_0^T|a_s^{n,m}|\,ds\Big)^2<\infty.
\]
Therefore applying Vitali's theorem gives
\[
\lim_{m\to\infty}\limsup_{n\to\infty}
\mathbb E\int_0^T(a_s^{n,m})^+\,ds=0.
\]
By using It\^o's formula and the fact that $Y^n\ge Y^{n,(m)}$, we deduce from the equation satisfied by $Y^n-Y^{n,(m)}$ that \cite[(2.18)]{RS2} holds with $|a_s^{n,m}|$
replaced by $(a_s^{n,m})^+$, which yields \cite[(2.19)]{RS2}. 
Moreover, since the terminal conditions coincide,
\[
\mathbb E\int_0^T a_s^{n,m}\,ds
=\mathbb E(Y_0^n-Y_0^{n,(m)})\ge0,
\]
and consequently
\[
\mathbb E\int_0^T|a_s^{n,m}|\,ds=2\mathbb E\int_0^T(a_s^{n,m})^+\,ds-\mathbb E\int_0^Ta_s^{n,m}\,ds
\le2\mathbb E\int_0^T(a_s^{n,m})^+\,ds.
\]
This also gives the convergence  needed for the martingale estimates.
The reduction to bounded data in Step~3 is unchanged.
\end{remark}

\begin{remark}
\label{rem2.6}
 Let $\tau\in\mathcal T$, with $\mathbb P(\tau>0)=1$.
In what follows with  any solution $(Y,M)$ of BSDE$^\tau(\zeta,f,X,\BG,\BP)$ we associate  the pairs $(Y^{[\tau]},M^{[\tau]})$ 
and $(Y^{(\tau)},M^{(\tau)})$ defined on $[0,\infty)$
by the formula 
\[
(Y^{(\tau)}_t,M^{(\tau)}_t):= (Y_{t\wedge \tau},M_{t\wedge \tau}),\quad (Y^{[\tau]}_t,M^{[\tau]}_t):= (Y_{t\wedge \tau}\fch_{[0,\tau)}(t) ,M_{t\wedge \tau}),\quad t\ge 0.
\]
We will frequently use  the following observation:
\[
Y^{[\tau]}_t=\int_t^\infty \fch_{\{s<\tau\}}f(X_s,Y^{[\tau]}_s)\,ds+\int_t^\infty\,dV_s-\int_t^\infty\,dM^{[\tau]}_s,\quad t\ge 0,
\]
where $V_t:= \fch_{\{t\ge \tau\}} \zeta,\, t\ge 0$. Consequently, for any $\sigma\in \mathcal T$, the pair 
$(Y^{[\tau]},M^{[\tau]})$  is a solution of $\operatorname{BSDE}^\sigma(Y^{[\tau]}_\sigma,\tilde f+dV,\tilde X,\mathbb G,\BP)$, where
\[
\tilde x=(x,x_{d+1}),\qquad \tilde f(\tilde x,y)=x_{d+1}\cdot f(x,y),\quad x_{d+1}\in\BR^+,\,x\in D,
\]
and
\[
X^{d+1}_t=\fch_{\{t<\tau\}},\quad \tilde X=(X,X^{d+1}).
\]
\end{remark}

\subsection{Stability results}

For $n\geq1$ let  $\BG^n=(\GG^n_t)_{t\ge0}$ be a filtration  in $\GG$ satisfying the usual hypotheses and $\TT^n$ be the set of finite-valued $\BG^n$-stopping times. 
We also assume as given a sequence of $\BG^n$-adapted c\`adl\`ag  processes $X^n$ such that $P(X^n_t\in D,\, t\in [0,\tau_n))=1$, where $\tau_n\in\TT^n$.

In what follows we use ``$\arrowpl$'' (resp. ``$\arrowdl$'') to denote the convergence in probability (resp. in distribution).
\begin{definition}
We say that $\BG^n$ converges weakly to $\BG$ on $[0,T]$, and we write $\BG^n\arroww\BG$, if  for every $B\in\GG_T$,
\[
\BE(\fch_B|\GG^n_{\cdot})
\arrowpl \BE(\fch_B|\GG_{\cdot})\quad\mbox{in }D([0,T];\BR).
\]
\end{definition}

\begin{theorem}
\label{th2.4}
Assume that \mbox{\rm(H2), (H3)} are satisfied and
\begin{enumerate}[\rm(a)]
\item $\tau_n\in\mathcal T^n$,  $\BE|g(X^n_{\tau_n})|<\infty$, $n\ge1$;
\item $\tau\in\mathcal T$,  $\BE|g(X_{\tau})|<\infty$.
\end{enumerate}
Let $(Y^n,M^n)$ be a solution of $\operatorname{BSDE}^{\tau_n}(g(X^n_{\tau_n}),f,X^n,\BG^n,\BP)$, and $(Y,M)$
be a solution of $\operatorname{BSDE}^{\tau}(g(X_{\tau}),f,X,\BG,\BP)$. Assume also that 
\begin{equation}
\label{eq2.2}
X^n\arrowpl X\mbox{ in }[D([0,T];\BR)]^d,\qquad \BG^n\arrowwl\BG\mbox{ on }[0,T]
\end{equation}
for every $T>0$, and  
\begin{equation}
\label{eq2.9}
\BP(\tau>0)=1,\quad \BE\tau<\infty,\quad \BE|\tau_n-\tau|\to 0,
\quad X^n_{\tau_n}\arrowpl X_{\tau},
\end{equation}		
\begin{equation}
\label{eq2.15}
\left(g(X^n_{\tau_n})\right)_{n\ge 1}\quad\text{ is  }\BP\text{-uniformly integrable}.
\end{equation}
Then $(Y^{n,(\tau_n)},M^{n,(\tau_n)})\arrowp(Y^{(\tau)},M^{(\tau)})$  in $D([0,T];\BR^2)$ for every $T>0$ such that $\BP(\tau=T)=0$ and $\BP(\Delta X_T=0)=1$.
\end{theorem}
\begin{proof} 
{\em Step 1.} We assume additionally that $g$ is bounded.
By  \cite[Theorem 2.7]{KR:JFA},
for any $n\geq1$  and $T>0$ there exists a unique solution $(Y^{n,T},M^{n,T})$  of  the problem 
BSDE$^{T\wedge\tau_n}(g(X^n_{T\wedge\tau_n}),f,X^n,\BG^n,\BP)$,
and a  unique solution $(Y^{T},M^{T})$  of  the problem
BSDE$^{T\wedge\tau}(g(X_{T\wedge\tau}),f,X,\BG,\BP)$. 
Set
\[
\tilde x=(x,x_{d+1}),\qquad \tilde f(\tilde x,y)=x_{d+1}\cdot f(x,y),\quad x_{d+1}\in\BR^+,\,x\in D,
\]
and
\[
X^{d+1}_t=\fch_{\{t<\tau\}},\quad \tilde X=(X,X^{d+1}),\quad X^{n,d+1}_t=\fch_{\{t<\tau_n\}},\quad \tilde X^n=(X^n,X^{n,d+1}).
\]
Then $\tilde f$ satisfies hypotheses \mbox{\rm(H2), (H3)}. Set
$\sigma_n= T\wedge\tau_n$ and $\sigma=T\wedge\tau$. Then
\[
 Y^{n,T,(\sigma_n)}_t=g(X^n_{T\wedge\tau_n})
+\int^{T}_{t}\tilde f(\tilde X^n_s, Y^{n,T,(\sigma_n)}_s)\,ds
-\int^{T}_{t}d M^{n,T,(\sigma_n)}_s,\quad t\in[0,T],
\]
and
\[
 Y^{T,(\sigma)}_t=g(X_{T\wedge\tau})
+\int^T_t\tilde f(\tilde X_s, Y^{T,(\sigma)}_s)\,ds-\int^T_td M^{T,(\sigma)}_s,  \quad t\in[0,T].
\]
Since $\tau_n\arrowp \tau$, $X^{n,d+1}\arrowp X^{d+1}$ in $D([0,T];\BR)$ for every $T>0$ such that $\BP(\tau=T)=0$. 
Since we also assume that $\BP(\Delta X_T=0)=1$, 
it follows from  the first part of assumption  (\ref{eq2.2}) and  \eqref{eq2.9} that
\[
\tilde X^n\arrowpl \tilde X\quad\mbox{in }[D([0,T];\BR)]^{d+1},\qquad X^n_{T\wedge\tau_n}\arrowpl X_{T\wedge\tau}\quad\mbox{in }\BR^d.
\]
By \cite[Theorem 2.6]{RS2}, for $T>0$ having the properties stated above,
\begin{equation}
\label{eq2.5}
( Y^{n,T,(\sigma_n)}, M^{n,T,(\sigma_n)})\arrowpl( Y^{T,(\sigma)}, M^{T,(\sigma)})\quad\mbox{in }D([0,T];\BR^2).
\end{equation}
By Proposition \ref{prop2.4} and Remark \ref{rem2.6}, for every $q\in (0,1)$,
\begin{align}
\label{eq2.6}
E\sup_{t\ge 0}|Y^{T,[\sigma]}_t- Y^{[\tau]}_t|^q&\le c_q \Big(\BE\int_{T\wedge \tau}^\tau|f(X_s,0)|\,ds+\BE(\fch_{\{\tau>T\}}(|g(X_{T\wedge \tau})|+|g(X_\tau)|))\Big)^q \nonumber\\
&\le c_q(\|f(\cdot,0)\|_\infty \BE(\tau-T)^++\|g\|_\infty P(T<\tau))^q
\end{align}
and
\[
\BE\int^{\tau}_{0}|f(X_s, Y^{[\tau]}_s)-f(X_s,Y^{T,[\sigma]}_s)|\,ds\le c (\|f(\cdot,0)\|_\infty \BE(\tau-T)^++\|g\|_\infty \BP(T<\tau)).
\]
From this, by using the formulas for $Y^{T,[\sigma]}$ and $Y^{[\tau]}$ (see Remark \ref{rem2.6}), we infer that for $q\in (0,1)$,
\begin{equation}
\label{eq2.8}
\lim_{T\rightarrow\infty}\BE\sup_{t\ge 0}|M^{T,(\sigma)}_t- M^{(\tau)}_t|^q=0.
\end{equation}
Analogously,
\begin{align}
\label{eq2.10}
&\BE\sup_{t\ge0}|Y^{n,T,[\sigma_n]}_t- Y^{n,[\tau_n]}_t|^{q}\le c_q(\|f(\cdot,0)\|_\infty \BE(\tau_n-T)^++\|g\|_\infty \BP(T<\tau_n))^q
\end{align}
and 
\begin{align*}
&\BE\int^{\tau_n}_{0}
|f(X^n_s,Y^{n,[\tau_n]}_s)-f(X^n_s,Y^{n,T,[\tau_n\wedge T]}_s)|\,ds\\
&\qquad\le c(\|f(\cdot,0)\|_\infty \BE(\tau_n-T)^++\|g\|_\infty \BP(T<\tau_n)).
\end{align*}
Furthermore, by  (\ref{eq2.9}),  
\begin{equation} 
\label{eq2.12}
\lim_{T\rightarrow\infty}\limsup_{n\rightarrow\infty}(\BP(T<\tau_n)+\BE(\tau_n-T)^+)=0.
\end{equation}
As a result we get
\[
\lim_{T\rightarrow\infty}\limsup_{n\rightarrow\infty}
\BE\Big(\sup_{t\ge 0}|Y^{n,T,[\sigma_n]}_t- Y^{n,[\tau_n]}_t|^q+
\int^{\tau_n}_{0}|f(X^n_s,Y^{n,[\tau_n]}_s)-f(X^n_s,Y^{n,T,[\sigma_n]}_s)|\,ds\Big)=0.
\]
From this, by using the formulas for $Y^{n,T,[\sigma_n]}$ and $Y^{n,[\tau_n]}$ (see Remark \ref{rem2.6}), we deduce that
\[
\lim_{T\rightarrow\infty}\limsup_{n\rightarrow\infty}
\BE\sup_{t\ge 0}|M^{n,T,(\sigma_n)}_t- M^{n,(\tau_n)}_t|^{q}=0.
\]
Observe that \eqref{eq2.6} and \eqref{eq2.10} hold with $[\sigma],[\tau]$ and $[\sigma_n],[\tau_n]$ replaced
by $(\sigma),(\tau)$ and $(\sigma_n),(\tau_n)$, respectively, with  possibly different constants. 
Consequently, these two estimates combined with   (\ref{eq2.5}), \eqref{eq2.8} and the above equality give the desired convergence.

{\em Step 2.} In this step, we dispense with the assumption that $g$ is  bounded. Let $\varepsilon>0$. Set
\[
g_k(x):= (g(x)\wedge k)\vee (-k),\quad x\in \BR^d.
\]
Let  $(Y^{n,k},M^{n,k})$ be the solution of BSDE$^{\tau_n}(g_k(X^n_{\tau_n}),f,X^n,\BG^n,\BP)$ and $(Y^k,M^k)$
be the solution of BSDE$^{\tau}(g_k(X_{\tau}),f,X,\mathbb G,\BP)$. Then, by Proposition \ref{prop2.4}, for any $q\in (0,1)$,
\[
\BE\sup_{t\ge 0}| Y^{n,[\tau_n]}_t- Y^{n,k,[\tau_n]}_t|^q
\le c_q\big(\BE(\fch_{\{|g(X^n_{\tau_n})|>k\}}|g(X^n_{\tau_n})|)\big)^q
\]
and 
\[
\BE\sup_{t\ge 0}| Y^{[\tau]}_t- Y^{k,[\tau]}_t|^q \le c_q\big(\BE(\fch_{\{|g(X_{\tau})|>k\}}|g(X_{\tau})|)\big)^q.
\]
Hence,  by \eqref{eq2.9} and \eqref{eq2.15}, there exists $k_0$ such that for every $k\ge k_0$ we have
\[
\BE\sup_{t\ge 0}| Y^{n,[\tau_n]}_t- Y^{n,k,[\tau_n]}_t|^q
+\BE\sup_{t\ge 0}| Y^{[\tau]}_t- Y^{k,[\tau]}_t|^q \le \varepsilon,\quad n\ge 1.
\]
By Proposition \ref{prop2.4},
\[
\BE\int^{\tau_n}_{0}|f(X^n_s, Y^{n,k,[\tau_n]}_s)-f(X^n_s, Y^{n,[\tau_n]}_s)|\,ds\le \BE(\fch_{\{|g(X^n_{\tau_n})|>k\}}|g(X^n_{\tau_n})|)
\]
and 
\[
\BE\int^{\tau}_{0}|f(X_s, Y^{k,[\tau]}_s)-f(X_s, Y^{[\tau]}_s)|\,ds\le \BE(\fch_{\{|g(X_{\tau})|>k\}}|g(X_{\tau})|).
\]
By this and  \eqref{eq2.9},  \eqref{eq2.15}, there exists $k_0$ such that for any $k\ge k_0$ and $n\ge 1$,
\[
\BE\int^{\tau}_{0}|f(X_s, Y^{k,[\tau]}_s)-f(X_s, Y^{[\tau]}_s)|\,ds+\BE\int^{\tau_n}_{0}|f(X^n_s, Y^{n,k,[\tau_n]}_s)-f(X^n_s, Y^{n,[\tau_n]}_s)|\,ds\le\varepsilon.
\]
From this and the previous step of the proof one easily concludes the result.
\end{proof}

\begin{remark}
A brief  inspection of the proof of Theorem \ref{th2.4} shows that the fourth  assumption in (\ref{eq2.9})  is superfluous when $g=0$.
\end{remark}

\begin{remark}
\label{rem2.8}
The sets $\{T\ge0: \mathbb P(\tau=T)>0\}$ and
$\{T\ge0:\BP(\Delta X_T\neq0)>0\}$ are at most countable.
The first statement is clear. For the second see, e.g., \cite[Lemma VI.3.12]{JSh}. If  $X$ is a  process with stationary  independent increments, 
then $\BP(\Delta X_T\neq0)=0$ for every $T>0$ (see \cite[II.4.3]{JSh}).
\end{remark}

\section{L\'evy processes}

To fix notation, we start with recalling some known results on L\'evy processes   to be used later on.

Let $b\in\BR^d$, $Q$ be a nonnegative definite symmetric \(d\times d\) matrix and $\nu$ be a L\'evy measure, i.e.  
a Borel measure on $\BR^d$ such that $\nu(\{0\})=0$ and $\int_{\BR^d}(1\wedge|y|^2)\,\nu(dy)<\infty$.
In what follows $X^0=(X^{0,1},\dots,X^{0,d})$ is a L\'evy process in $\BR^d$ defined on some complete probability space $(\Omega_0,\GG,\BP)$
with the generating  triplet $(Q,\nu,b)$, i.e. with the characteristic function 
\[
\BE e^{i(\xi,X^0_t)}= e^{-t\psi(\xi)},\quad t>0,\quad \xi\in\BR^d,
\]
where  the  characteristic exponent $\psi:\BR^d\rightarrow\BC$ is  given by \eqref{eq1.4}
(see \cite[Theorem 8.1, Corollary 11.6]{Sa}). By $\BG:=(\mathcal G_t)_{t\ge0}$ we denote the completion of the  filtration generated by $(X^0_t)$,
i.e. $\mathcal G_t=\sigma(X^0_s,\, s\le t)\vee\mathcal N$, where $\mathcal N$ is the set of all $\BP$-null sets in $\mathcal G$.
Such a filtration is right-continuous (see \cite[Theorem I.32]{P}).
By \cite[Theorem 31.5]{Sa}, 
\[
P_t\eta(x):=\BE\eta(x+X^0_t),\quad \eta\in C_0(\BR^d),\,\, t\ge 0,
\]
is a $C_0$-semigroup on $C_0(\BR^d)$ and its generator 
$L$ restricted   to $C^2_c(\BR^d)$ can be written as the integro-differential operator  (\ref{eq1.2}). 

It is known that if the characteristic exponent  \eqref{eq1.4} satisfies   \eqref{eq1.3}, then 
\begin{equation}
\label{eq3.1}
\Psi(u,v):=\int_{\BR^d}\hat u(\xi)\overline{\hat v(\xi)}\psi(\xi)\,d\xi, \quad u,v\in H^{\psi},
\end{equation}
where
\begin{equation}
\label{eq3.2}
H^{\psi}=\Big\{u\in L^2(\BR^d):\int_{\BR^d}(1+|\psi(\xi)|)|\hat u(\xi)|^2\,d\xi<\infty\Big\}
\end{equation}
($\hat u$ denotes the Fourier transform of $u$)
is a (nonsymmetric) regular Dirichlet form  on $L^2(\BR^d)$ (see \cite{BergForst1973}). 
For $x\in \BR^d$ we set 
\[
X^x=x+X^0,\qquad\tau^x=\inf\{t>0:X^x_t\notin D\},\qquad \bar\tau^x=\inf\{t>0:X^x_t\notin\bar D\}.
\]
We define the transition kernels $P_t(x,\,dy)$ and the resolvent kernels $R_{\alpha}(x,dy)$ of $X^0$ by setting
\[
P_t(x,B)=\BP(X^x_t\in B),\quad t>0,\qquad 
R_{\alpha}(x,B)=\BE\int^{\infty}_0e^{-\alpha t}\fch_B(X^x_t)\,dt,\quad \alpha\ge0,
\]
for any $x\in\BR^d$ and Borel set $B\subset\BR^d$. We say that $X^0$ satisfies 
condition (ACP) ({\em absolute continuity of potential measures})   
if 
\[
R_{\alpha}(x,dy)\ll m\quad\mbox{for all $\alpha>0$ and $x\in \BR^d$},
\]
where $m$ denotes the  Lebesgue measure on $\BR^d$, and we say that $X^0$ 
satisfies condition (ACT) ({\em absolute continuity of transition measures}) if
\[
P_t(x,dy)\ll m\quad\mbox{for all $t>0$ and $x\in \BR^d$},
\]
From the definition of $R_{\alpha}$ it follows that (ACT) implies (ACP). However, (ACP) 
does not imply (ACT). For a counterexample see, e.g., \cite[Remark 41.13]{Sa} or \cite[Example 4.2.1]{FOT}.  
By \cite[Theorem 4.2.4]{FOT}, conditions (ACP) and (ACT) are equivalent for symmetric processes $X$. The  Hartman--Wintner  condition \eqref{eq1.9}
is sufficient for (ACT) and hence (ACP).

For any $x\in D$ and a Borel subset $B$ of $\BR^d$ we set
\[
P_D(x,B)=\BP(X^x_{\tau^x}\in B),\quad x\in D.
\]
The kernel $P_D(x,dy)$ is called  the {\em Poisson kernel} for $L$ and $D$. In  general, the measure $P_D(x,dy)$ is concentrated on $D^c$
 but for continuous processes $X^0$ and finite stopping times $\tau^x$ it is concentrated on $\partial D$.  We also consider the Green kernel 
 given by 
  \[
   G_D(x,B):=\mathbb E\int_0^{\tau^x_D}\mathbf1_B(X^x_t)\,dt.
   \]
 Suppose that $X^0$ is  non-trivial 
 (i.e. at least one component of its triplet is non-zero)
and that condition (ACP) is satisfied. Then 
there exists a nonnegative  jointly measurable function $G_D:D\times D\rightarrow\BR$, called the {\em Green function} for $L$ and  $D$, such that 
\begin{equation}
\label{eq3.3}
\mathbb E\int^{\tau^x}_0h(X^x_t)\,dt=\int_Dh(y)G_D(x,y)\,dy,\quad x\in D,\quad h\in\BB_b(\BR^d).
\end{equation}
To see this, we consider the resolvent $\tilde R_{\alpha}$ associated with the dual process of $X^0$ (see \cite[Definition 41.6]{Sa}). 
It satisfies (ACP). Therefore  the part $R^D_{\alpha}$ of $R_{\alpha}$ on $D$ and the part  $\tilde R^D_{\alpha}$ of $\tilde R_{\alpha}$ on $D$
(see \cite[Definition 42.1]{Sa}) satisfy (ACP). By Hunt's switching formula (see \cite[Proposition 42.3]{Sa}) and \cite[Theorem VI.(1.4)]{BG}, 
the resolvents $R^D_{\alpha}$ and $\tilde R^D_{\alpha}$ are dual with respect to $m$ in the sense defined  in \cite[Definition VI.(1.2)]{BG}. 
The existence of $G_D$ now follows from the remarks in \cite[p. 256]{BG}. If $X^0$ is symmetric, then $G_D$ has the property that $G_D(x,y)=G_D(y,x)$. 
Note also that for symmetric processes the existence of $G_D$ follows from \cite[Lemma 4.2.4]{FOT} applied to the part $\Psi^D$ of the form 
$\Psi$ on $D$ and an approximation argument (we set $G_D=\lim_{\alpha\downarrow0}r_{\alpha}$, where $r_\alpha$ is the resolvent density 
of \cite[Lemma 4.2.4]{FOT} associated with $\Psi^D$).

From (\ref{eq3.3}) it follows in particular that 
\[
\mathbb E\tau^x=\int_DG_D(x,y)\,dy, \quad x\in D.
\]

\subsection{Canonical realization of a L\'evy process}

Throughout the paper, we shall frequently use the canonical realization of a given 
L\'evy process  $X^0=(X^0_t)_{t\geq0}$ (see \cite[Section 40]{Sa}). More precisely, we set 
\[
\Omega=D([0,\infty);\BR^d),
\]
i.e. $\Omega$ is the Skorokhod  space of all c\`adl\`ag functions from $[0,\infty)$ into $\BR^d$ equipped with the $J_1$-topology.  
We denote by $\mathcal A$ the  Borel $\sigma$-field of subsets of $\Omega$. Let
\[
X_t(\omega)=\omega(t),
\qquad \omega\in\Omega,\quad t\geq 0,
\]
i.e. $X$ is the canonical process on $\Omega$.  
Observe that for every $x\in\BR^d$,  $X^x$ is  an $\Omega$-valued random
variable. We define a probability measure $\mathbb P_x$ on
$(\Omega,\mathcal A)$ as the law of $X^x$, that is
\[
\BP_x(B)=
\mathbb P(X^x\in B),\quad B\in\mathcal A.
\]
Equivalently, 
\[
\BP_x=\BP_0\circ s_x^{-1},
\]
where $\BP_0$ is the law of $X^0$ on $\Omega$ and
\[
s_x:\Omega\longrightarrow\Omega,
\qquad
s_x(\omega)(t)=x+\omega(t).
\]
Under the measure $\BP_x$ the canonical process $X$ starts from $x$ and
$(X_t-x)_{t\geq0}$
is a L\'evy process with the same law as $X^0$. We write $\BE_x$ for the expectation with respect to $\mathbb P_x$.
Unless otherwise stated, $\BF=(\FF_t)_{t\geq0}$ denotes the
right-continuous augmentation of the canonical filtration
\[
\FF_t^0=\sigma(X_s:0\leq s\leq t),
\qquad t\geq0.
\]
It is well known that  $X$ is a  strong Markov process with respect to $\BF$ (see  \cite[Chapter~I, Proposition~6]{Bertoin1996}).

\subsection{Capacity-zero sets and $m$-inessential sets}

Let $m$ denote the Lebesgue measure
on $\mathbb R^d$. The translation invariance of $(X_t)$ implies that $m$ is invariant
for the  semigroup $(P_t)_{t\geq 0}$ of $X$.
In particular, $m$ is an excessive measure for $X$.

Fix \(\alpha>0\) and choose \(\rho\in L^1(\mathbb R^d)\) such that
\[
\rho>0\quad m\text{-a.e.},
\qquad
\int_{\BR^d} \rho(x)\,m(dx)=1
\]
Set $\mu=\rho m$. Then  $\mu$ is a probability measure equivalent to
$m$. For a Borel set $A\subset \BR^d$ define
\[
e_A^\alpha(x)=
\BE_x\!\left[e^{-\alpha T_A};\,T_A<\infty\right],
\quad x\in \BR^d,
\]
where $T_A$ is the first hitting time of $A$, i.e.
\[
T_A=\inf\{t>0:X_t\in A\}.
\]
We shall use the following $\alpha$-hitting capacity:
\[
\operatorname{Cap}_{\mu}^{\alpha}(A)
:=
\int_{\BR^d} e_A^\alpha(x)\,\mu(dx).
\]
Let $A\subset \BR^d$ be a Borel set. We say that $A$ is $m$-polar if
\[
\BP_x(T_A<\infty)=0
\quad\text{for $m$-almost every }x\in \BR^d.
\]
Since $\mu$ and $m$ are equivalent, it follows directly from the
definition of $\operatorname{Cap}_{\mu}^{\alpha}$ that
\[
\operatorname{Cap}_{\mu}^{\alpha}(A)=0
\quad\Longleftrightarrow\quad
e_A^\alpha=0
\quad m\text{-a.e.}
\quad\Longleftrightarrow\quad
\mathbb P_x(T_A<\infty)=0
\quad m\text{-a.e.}
\]
Consequently,
\[
\operatorname{Cap}_{\mu}^{\alpha}(A)=0
\quad\Longleftrightarrow\quad
A\text{ is $m$-polar}.
\]
A nearly Borel set $N\subset \BR^d$ is called $m$-inessential if $N$ is
$m$-polar and $\BR^d\setminus N$ is $X$-absorbing, i.e.
\[
\BP_x\bigl(X_t\in \BR^d\setminus N\text{ for every }t\geq 0\bigr)=1,
\quad x\in \BR^d\setminus N.
\]

\begin{proposition}
\label{prop3.1}
Let $B\subset \BR^d$ be a Borel set. If
$\operatorname{Cap}_{\mu}^{\alpha}(B)=0$,
then there exists a Borel $m$-inessential set $N\subset \BR^d$ such that $B\subset N$.
\end{proposition}

\begin{proof}
By the preceding discussion, the condition
$\operatorname{Cap}_{\mu}^{\alpha}(B)=0$ implies that $B$ is
$m$-polar. It follows from
\cite[Proposition~6.12]{GetoorSharpe1984} (see also
\cite[Theorem~A.2.15]{ChenFukushima2012}) that every $m$-polar set is
contained in a Borel $m$-inessential set. Hence there exists a Borel
$m$-inessential set $N$ such that $B\subset N$.
\end{proof}

\begin{remark}
\label{rem3.2}
Let $(P^D_t)_{t\ge0}$ denote the semigroup determined by the process $X$ killed upon leaving $D$.
Recall that a Borel function $u:D\to [0,\infty]$ is called $(P^D_t)$-excessive 
if $\sup_{t>0} P^D_tu(x)=u(x),\, x\in D$. By 
\cite[Theorem III.5.7]{BG}, if $u$ is a $(P^D_t)$-excessive function  that is finite $m$-a.e. in $D$,
then the set $\{u=\infty\}$ is $m$-inessential.  The important examples of $(P^D_t)$-excessive functions are: $D\ni x\mapsto \mathbb E_x\gamma(X_{\tau_D})$
for any Borel measurable function $\gamma :\BR^d\to [0,\infty]$ (see, e.g., \cite[Proposition II.2.8]{BG}), 
and $D\ni x\mapsto R^D\eta(x):=\mathbb E_x\int^{\tau_D}_0\eta(X_t)\,dt$ 
for any Borel measurable function $\eta:\BR^d\to [0,\infty]$ (see, e.g., \cite[Proposition II.2.2]{BG}).
\end{remark}

\subsection{L\'evy systems}

It is known that as a L\'evy system $(N,H)$ for $(X,\BP_x)$ (see, e.g., \cite[Definition A.3.7]{FOT}) one may take
\begin{equation}
	\label{eq6.1}
	N(x,dy)=\nu(dy-x),\quad x\in\BR^d,\qquad H_t=t,\quad t\ge0 
\end{equation}
(see \cite[p. 343]{Sh}). Therefore for any nonnegative Borel measurable function $f:\BR^d\times\BR^d\rightarrow\BR$ 
with $f(x,x)=0$ for $x\in\BR^d$ and any stopping time $\tau$ we have
\begin{equation}
\label{eq3.5}
\BE_x\sum_{s \leq \tau} f(X_{s-}, X_s)
=\BE_x\int_0^\tau\!\!\int_{\BR^d} f(X_s, y)\,\nu(d y-X_s)\,ds,\quad x\in\BR^d.
\end{equation}

As an application of (\ref{eq3.5}) we get the following proposition, which looks rather technical, but in fact play significant role 
in the proof of our main Theorem \ref{th3.4} and  the proof of important Proposition \ref{prop6.5}.

\begin{proposition}
\label{prop4.3}
Let $D\subset\BR^d$ be an open bounded set and
$X$ be a L\'evy process on $\BR^d$. Then  
	\begin{equation}
		\label{eq6.9}
		\mathbb P_x\left (\exists\, s\in (0,\tau_D], \Delta X_s\neq0, X_{s-}\in\partial D\right )=0.	
	\end{equation}
\end{proposition}
\begin{proof}
	To simplify notation, we set $\tau=\tau_D$. 
	Let $x \in D$. For $k \geq 1$ define
	\[
	N_k=\sum_{0<s \leq \tau} \mathbf{1}_{\left\{X_{s-} \in \partial D\right\}} \mathbf{1}_{\left\{\left|\Delta X_s\right|>1 / k\right\}}.
	\]
	Then
	\begin{align*}
		\mathbb{E}_x N_k & =\mathbb{E}_x \int_0^\tau \int_{\mathbb{R}^d} \mathbf{1}_{\partial D}\left(X_s\right) \mathbf{1}_{\{|z|>1 / k\}} \nu(d z) d s \\
		& =\nu(\{|z|>1 / k\}) \mathbb{E}_x \int_0^\tau \mathbf{1}_{\partial D}\left(X_s\right) d s=0
	\end{align*}
	which proves (\ref{eq6.9}).
\end{proof}

\section{Probabilistic, Sobolev and viscosity  solutions to PDEs}
\label{sec3}

Since probabilistic methods play a central role in the proofs of our stability results, 
we adopt the notion of probabilistic solution as the primary notion of a solution of (\ref{eq1.1}). 
This choice is also motivated by its generality. Namely, its definition does not impose any additional structural assumptions on either the operator $L$ or the set $D$.
We shall show that, under natural assumptions on the data and for the class of operators considered in this paper, 
Sobolev and viscosity solutions are probabilistic solutions.
In what follows, we consider only non-zero operators $L$.

\subsection{Probabilistic solutions}

\begin{definition}
We say that a Borel measurable  function $u:\mathbb R^d\rightarrow\BR$ is a {\em probabilistic solution} of (\ref{eq1.1}) if 
there exists an $m$-inessential set $N\subset \mathbb R^d$ such that 
$\BE_x|g(X_{\tau_D})|<\infty$ and  $\BE_x\int_{0}^{\tau_D}|f(X_{t},u(X_t))|\,dt<\infty$
for every  $x\in D\setminus N$, and 
\begin{equation}
\label{eq4.1}
u(x)=\BE_xg(X_{\tau_D})
+\BE_x\int_{0}^{\tau_D}f(X_{t},u(X_t))\,dt, \quad x\in D\setminus N,
\end{equation}
\begin{equation}
\label{eq4.2}
u(x)=g(x),\quad x\in D^c\setminus N.
\end{equation}
The set $N$ is called an exceptional set for  $u$.
\end{definition}

For the reader's convenience, in the next theorem we recall  known connections between  
solutions of (\ref{eq4.1}), (\ref{eq4.2}) and solutions of Markov-type BSDEs.

\begin{theorem}
\label{th4.2}
Let $g:D^c\rightarrow\BR$  and $f:D\times\BR\to\BR$ be  measurable functions.
\begin{enumerate}
\item[\rm(i)]  If $u$ is a probabilistic solution to \eqref{eq1.1} with exceptional set $N$, then for every $x\in D\setminus N$ the pair  $(Y,L^x)$, where
\[
Y_t=u(X_{t\wedge\tau_D}),\quad L^x_t= \BE_x\Big(g(X_{\tau_D})+\int_0^{\tau_D}f(X_s,u(X_s))\,ds\,\big|\,\FF_t\Big)-u(x),\quad  t\ge 0,
\]
is a solution of  $\operatorname{BSDE}^{\tau_D}(g(X_{\tau_D}),f,X,\mathbb F,\mathbb P_x)$, and  the pair $(Y^x_t,M^x_t)$, where
\[
Y^x_t=u(X^x_{t\wedge\tau^x_D}),\quad M^x_t= \BE\Big(g(X^x_{\tau^x_D})
+\int_0^{\tau^x_D}f(X^x_s,u(X^x_s))\,ds\,\big|\,\GG_t\Big)-u(x),\quad t\ge 0,
\]
 is a solution of   $\operatorname{BSDE}^{\tau^x_D}(g(X^x_{\tau^x_D}),f,X^x,\BG,\BP)$.

\item[\rm(ii)]
If {\rm(H3)} holds, then the probabilistic solution of
\eqref{eq1.1} is unique up to an $m$-inessential set.

\item[\rm(iii)] If \mbox{\rm (H2), (H3)} hold and $\int_{D^c} |g(z)|\,P_D(x,dz)<\infty$   for a.e. $x\in D$, then there exists a probabilistic solution of \eqref{eq1.1}.
\end{enumerate}
\end{theorem}
\begin{proof}
The first assertion in (i)  easily follows from the Markov property of $(X_t)$.
The second one follows then from the fact that $Y$ and $Y^x$ have the same law. Indeed, we have
\[
u(X_{t\wedge\tau_D})=\BE_x\Big(g(X_{\tau_D})+\int_{t\wedge\tau_D}^{\tau_D} f(X_s,u(X_s))\,ds\,\big|\,\FF_t\Big).
\]
Hence,  for any Borel sets $B_1,\dots, B_n$ and $0\le t_1,\dots,t_n\le t$, 
\begin{align*}
&\BE_x(\fch_{\{X_{t_1}\in B_1,\dots,X_{t_n}\in B_n\}}u(X_{t\wedge\tau_D}))\\
&\qquad=
\BE_x\Big(\fch_{\{X_{t_1}\in B_1,\dots, X_{t_n}\in B_n\}}(g(X_{\tau_D})+\int_{t\wedge\tau_D}^{\tau_D} f(X_s,u(X_s))\,ds)\Big).
\end{align*}
Since the law of $X$ under $\BP_x$ equals the law of $X^x$ under $\BP$, we deduce from the above equation that
\begin{align*}
&\BE(\fch_{\{X^x_{t_1}\in B_1,\dots, X^x_{t_n}\in B_n\} }u(X^x_{t\wedge\tau_D^x}))\\
&\qquad=\BE\Big(\fch_{\{X^x_{t_1}\in B_1,\dots, X^x_{t_n}\in B_n\}}(g(X^x_{\tau^x_D})
+\int_{t\wedge \tau^x_D}^{\tau^x_D} f(X^x_s,u(X^x_s))\,ds)\Big).
\end{align*}
From this and the fact that $\BG$ is the completion of the natural filtration generated by $(X^0_t)$ one easily obtains
\[
u(X^x_{t\wedge \tau^x_D})=\BE\Big(g(X^x_{\tau^x_D})+\int_{t\wedge \tau^x_D}^{\tau^x_D} f(X^x_s,u(X^x_s))\,ds\,\big|\,\GG_t\Big),
\]
which in turn implies the second assertion of (i). Assertion (ii) follows from Theorem \ref{th2.2}.
As for (iii), first note that, by Remark \ref{rem3.2} and the monotonicity of $f$, 
there exists an $m$-inessential set $N\subset \mathbb R^d$ such that
$\int_{D^c} |g(z)|\,P_D(x,dz)<\infty$ for $x\in D\setminus N$.
Then, by Theorem \ref{th2.2},  for any $x\in D\setminus N$ there exists a unique solution 
$(Y^x,M^x)$ to $\operatorname{BSDE}^{\tau_D}(g(X_{\tau_D}),f,X,\mathbb F,\BP_x)$.
By \cite[Proposition 3.5]{KR:JFA}, $Y^x_t=u(X_t),\, t\in [0,\tau_D]$,  $\BP_x$-a.s. for 
$x\in D\setminus N$, where  $u(x):=\BE_x Y^x_0,\, x\in D\setminus N$. This completes the proof of (iii).
\end{proof}

\subsection{Sobolev space  solutions}

Throughout  this subsection, we assume additionally that the sector condition  \eqref{eq1.3} holds and $g\in H^\psi$.
Furthermore, we assume that 
\begin{equation}
\label{eq4.4}
1/|\psi(\xi)|\quad\mbox{is locally integrable on }\BR^d,\qquad \lim_{|\xi|\rightarrow\infty}|\psi(\xi)|=\infty,
\end{equation}
with the convention that $1/|\psi(\xi)|=\infty$ for $\xi$ such that $\psi(\xi)=0$. This
implies that  $(\Psi, H^\psi)$ is  transient (see \cite[Definition 3.5.19.]{J2} and \cite[Corollary 37.6]{Sa}).

We denote by $(T^D_t)$ the $C_0$-semigroup on $L^2(D)$ related to the form $(\Psi,H^\psi_0(D))$, that is the part of 
$(\Psi,H^\psi)$ on $D$, and we denote by   $(G^D_\alpha)$ the corresponding resolvent (see \cite[Theorem 3.5.7]{Oshima}). 
Note that
\[
H_0^\psi(D)=\{v \in H^\psi: \tilde{v}=0 \text{ q.e. on  } D^c\},
\]
where $\tilde v$ is a quasi-continuous version of $v$. 

\begin{remark}
\label{rem4.3}
We shall use the following standard consequence of coercivity of $\Psi$ (second condition in \eqref{eq4.4}): the embedding 
\[
 H_0^\psi(D) \hookrightarrow L^2(D)
\]
is compact and there exists a
constant \(C>0\) such that
\[
 \|v\|_{L^2(D)}^2 \leq C\,\Psi(v,v),\quad v\in H^\psi_0(D).
\]
To see this, consider a sequence  $(u_n)$ bounded in \(H_0^\psi(D)\). Since
\(u_n=0\) q.e. on \(D^c\), in particular \(u_n=0\) a.e. on \(D^c\). As \(D\) is
bounded, the sequence is tight in $L^2(\mathbb R^d)$. It remains to check the
uniform continuity with respect to translations. By Plancherel's theorem,
\[
 \|u_n(\cdot+h)-u_n\|_{L^2(\mathbb R^d)}^2
 =
\int_{\BR^d}|e^{ih\cdot \xi}-1|^2 |\widehat u_n(\xi)|^2\,d\xi,
\]
and consequently,
\begin{align*}
 \|u_n(\cdot+h)-u_n\|_{L^2(\BR^d)}^2
 &\leq \int_{|\xi|\le R}
 |e^{ih\cdot \xi}-1|^2 |\widehat u_n(\xi)|^2\,d\xi\\
 &\quad+\frac{4}{\inf_{|\xi|>R}|\psi(\xi)|}
 \int_{\mathbb R^d} |\psi(\xi)| |\widehat u_n(\xi)|^2\,d\xi.
\end{align*}
Since \(|\psi(\xi)|\to\infty\) as $|\xi|\rightarrow\infty$, the second term on the  right-hand side of the above inequality is uniformly small in \(n\) for \(R\) sufficiently large,
while the first term tends to zero uniformly in \(n\) as \(h\to0\). Hence, by
the Fr\'echet--Kolmogorov compactness criterion, \((u_n)\) is relatively compact
in \(L^2(\BR^d)\), and therefore also in \(L^2(D)\). Now,  a standard contradiction argument gives the asserted Poincar\'e inequality.
\end{remark}

For \(\alpha>0\), we write 
\[
\Psi_\alpha(u,v)=\Psi(u,v)+\alpha(u,v)_{L^2(\mathbb R^d)}.
\]
Define
\[
\mathscr{H}_D^\alpha=\{v \in H^\psi: \Psi_\alpha(v, w)=0 \text { for all } w \in H^\psi_0(D)\}.
\]
By \cite[Theorem 2.4.2]{Oshima}, for any $u\in H^\psi$ there exist unique  
$h\in \mathscr{H}_D^\alpha$ and $v\in H^\psi_0(D)$ such that
\[
u=h+v.
\]
We let $H^\alpha_Du:=h$. Note that $H^\alpha_Du=u$ q.e. on $D^c$.
For any $u\in H^\psi$ and quasi-continuous Borel function $v$ that belongs to the  $m$-equivalence class of $u$ the function 
$\tilde u(x):=\BE_x[e^{-\alpha\tau_D}v(X_{\tau_D})]$,
whenever $\BE_x[e^{-\alpha\tau_D}|v(X_{\tau_D})|]<\infty$ and $\tilde u(x):=0$ otherwise, is a quasi-continuous version of 
$H_D^\alpha u$ (see \cite[Theorem 3.5.1]{Oshima}).
Recall also that for any $h\in L^2(D)$ and any Borel function $h_0$ that belongs to the $m$-equivalence class of $h$ the function 
$\tilde u(x):=\BE_x\int_0^{\tau_D} e^{-\alpha s} h_0(X_s)\,ds$,
whenever $\BE_x\int_0^{\tau_D} e^{-\alpha s} |h_0(X_s)|\,ds<\infty$ and $\tilde u(x):=0$ otherwise, is a quasi-continuous version of $G^D_\alpha h$
(see \cite[Theorem 3.3.4]{Oshima}).

\begin{definition}
We say that  $u\in H^\psi$ is a weak solution of \eqref{eq1.1} if $u-g\in H^\psi_0(D)$, $f(\cdot,u)\in L^2(D)$
and 
\[
\Psi(u,v)=\int_Df(x,u(x))v(x)\,dx, \quad v\in H^{\psi}_0(D).
\]
\end{definition}

\begin{proposition}
If $u\in H^\psi$ is a weak solution of \eqref{eq1.1}, then for any quasi-continuous Borel measurable  $m$-version $g_0$ of $g$
there exists a Borel version $\tilde u$ of $u$ which is a probabilistic solution of \eqref{eq1.1} with $g=g_0$.
\end{proposition}
\begin{proof}
If $g\in H^{\psi}(\BR^d)$, then $H^\alpha_Dg\in H^{\psi}(\BR^d)$ and 
$\Psi_\alpha(H^\alpha_Dg,v)=0$ for $v\in H^{\psi}_0(D)$.  
Hence, for every $v\in H^{\psi}_0(D)$,
\[
\Psi_\alpha(u-H^\alpha_Dg,v)=\int_Df(x,u(x))v(x)\,dx+\alpha\int_D u(x)v(x)\,dx.
\]
Thus $u-H^\alpha_Dg=G^D_\alpha f(\cdot,u)+\alpha G^D_\alpha u$. Let $u_0$ be a Borel measurable and quasi-continuous 
function that belongs to the $m$-equivalence class of $u$ (see \cite[Theorem 3.5.7]{Oshima}). We may assume that $g$ is quasi-continuous
and Borel measurable.
Then, by the comments preceding the proposition,   
\[
u_0(x)=\BE_xe^{-\alpha\tau_D}g(X_{\tau_D})+\BE_x\int_0^{\tau_D} e^{-\alpha s} f(X_s,u_0(X_s))\,ds+\alpha \BE_x\int_0^{\tau_D}  e^{-\alpha s} u_0(X_s)\,ds
\] 
for a.e. $x\in D$. Letting $\alpha\rightarrow0$ we get that 
\begin{equation}
\label{eq4.3}
u_0(x)=\BE_xg(X_{\tau_D})+\BE_x\int_0^{\tau_D}f(X_s,u_0(X_s))\,ds
\end{equation}
for a.e. $x\in D$ provided we show that the set 
\[
N_0:=D\setminus \Big\{x\in D:\BE_x\Big(|g(X_{\tau_D})|+\int_0^{\tau_D}|f(X_s,u_0(X_s))|\,ds
+\int_0^{\tau_D}|u_0(X_s)|\,ds\Big)<\infty \Big\}
\]
is of Lebesgue measure zero. By the Poincar\'e inequality (see Remark \ref{rem4.3}),
\begin{align*}
&\Big\|\BE_\cdot \int_0^{\tau_D} e^{-\alpha s} |f(X_s,u_0(X_s))|\,ds+\BE_\cdot\int_0^{\tau_D}  e^{-\alpha s} |u_0(X_s)|\,ds\Big\|_{L^2(D)}\\
&\qquad\le c(\|f(\cdot,u_0)\|_{L^2(D)}+\|u_0\|_{L^2(D)}).
\end{align*}
On the other hand, by transience of $\Psi$ (see \cite[Theorem 3.5.33]{J2} and \cite[Lemma 2.4.1]{Oshima}),
for some strictly positive function $\rho$ we have
\[
\int_{\BR^d}\BE_x e^{-\alpha\tau_D}|g(X_{\tau_D})| \rho(x)\,dx
\le \sqrt{\Psi_\alpha(H^\alpha_D|g|,H^\alpha_D|g|)}\le c\sqrt{\Psi_\alpha(|g|,|g|)}\le c\sqrt{\Psi_\alpha(g,g)}.
\]
Consequently, by Fatou's lemma, $m(N_0)=0$. Now, by Remark \ref{rem3.2} and \cite[Theorem 3.2.9, Theorem 3.5.6, Corollary 3.5.3]{Oshima}, 
we infer that \eqref{eq4.3}
holds q.e. in $D$. Since $H^\alpha_Dg=g$ q.e. in $D^c$, we obtain that  $u=g$ q.e. in $D^c$.
Finally, by Proposition \ref{prop3.1} we get the result.
\end{proof}

\subsection{Viscosity solutions}

\begin{definition}
We say that $u\in C_b(\mathbb R^d)$ is a viscosity solution of the problem
\begin{equation}
\label{eq4.6}
\begin{cases}
-Lu=h & \text{in }D,\\
u=g & \text{in }D^c
\end{cases}
\end{equation}
if $u=g$ on $D^c$ and  for any $x_0\in D$ and any $\varphi\in C_b^2(\BR^d)$
 such that $u-\varphi$ attains a global maximum at $x_0$ we have 
\[
-L\varphi(x_0)\leq h(x_0)
\]
while if  $u-\varphi$ attains a global minimum at $x_0$, then
\[
-L\varphi(x_0)\geq h(x_0).
\]
\end{definition}

Consider a nonnegative function  $\rho\in C_c^\infty(\mathbb R^d)$  supported in
the unit open ball $B(0,1)$ such that
\[
\int_{\BR^d}\rho(y)\,dy=1.
\]
For $\varepsilon>0$, put
\[
\rho_\varepsilon(y):=\varepsilon^{-d}\rho(y/\varepsilon).
\]

\begin{lemma}
\label{lem4.7}
Let $U$ be an open subset of $\BR^d$ and $u,h\in C_b(\BR^d)$. Assume that $u$ is a viscosity solution of
\[
-Lu=h\quad\text{in }U.
\]
For $\varepsilon>0$ put  $u_\varepsilon=u*\rho_\varepsilon$, $h_\varepsilon=h*\rho_\varepsilon$ and define
\[
U_\varepsilon=
\{x\in U:\operatorname{dist}(x,U^c)>\varepsilon\}.
\]
Then
\[
-Lu_\varepsilon=h_\varepsilon
\quad\text{pointwise in }U_\varepsilon.
\]
\end{lemma}

\begin{proof}
For $y\in B(0,\varepsilon)$ define
$u^y(x):=u(x-y)$, $h^y(x):=h(x-y)$.
Since $L$ is translation invariant, $u^y$ is a viscosity solution of
\[
-Lu^y=h^y
\]
in the translated domain $U+y$. In particular, $u^y$ is a viscosity
solution of this equation in $U_\varepsilon$.
Since the  equation is linear,  
finite convex combinations of viscosity solutions are viscosity solutions
(by the nonlocal Jensen–Ishii lemma, see \cite[Lemma 1]{BarlesImbert2008}).  Approximating the
integral defining $h_\varepsilon$ by finite convex combinations of the
functions $h^y$ and using the stability of viscosity inequalities under
locally uniform convergence (see \cite[Theorem~1]{BarlesImbert2008}), we conclude that $u_\varepsilon$ is a
viscosity solution of
\[
-Lu_\varepsilon=h_\varepsilon
\quad\text{in }U_\varepsilon.
\]
From this one readily obtains the desired result.
\end{proof}

\begin{theorem}
Let $h,g\in C_b(\BR^d)$ and $u\in C_b(\BR^d)$ be a viscosity solution of \eqref{eq4.6}.
Then 
\[
u(x)=\BE_x\Big\{g(X_{\tau_D})+\int_0^{\tau_D}h(X_s)\,ds\Big\},
\quad x\in D.
\]
\end{theorem}

\begin{proof}
For $n\ge 1$ let
\[
D_n=\{x\in D:\operatorname{dist}(x,D^c)>2/n\}.
\]
Then $D_n\subset\subset D$ and  $D_n\uparrow D$. 
Fix $x\in D_n$ and $0<\varepsilon<1/n$. By Lemma~\ref{lem4.7},
\[
-Lu_\varepsilon=h_\varepsilon\quad \text{pointwise in }D_\varepsilon.
\]
Since $D_n\subset D_\varepsilon$, Dynkin's formula (see \cite[Theorem~3.9.4, pp.~145--146]{Kolokoltsov2011}) yields
\[
\BE_xu_\varepsilon(X_{\tau_{D_n}})-u_\varepsilon(x)=
\BE_x\int_0^{\tau_{D_n}}Lu_\varepsilon(X_s)\,ds=
-\mathbb E_x
\int_0^{\tau_{D_n}}h_\varepsilon(X_s)\,ds.
\]
Letting $\varepsilon\downarrow0$ and using the dominated convergence theorem we get
\begin{equation}
\label{eq4.9}
u(x)=\BE_x\Big\{u(X_{\tau_{D_n}})+\int_0^{\tau_{D_n}}h(X_s)\,ds\Big\}.
\end{equation}
Since $\tau_{D_n}\uparrow\tau_D$, it follows from the quasi-left-continuity of $X$ (see \cite[Chapter~I, Proposition~7, p.~21]{Bertoin1996}) that
\[
X_{\tau_{D_n}}\rightarrow X_{\tau_D}
\quad\text{almost surely}.
\]
Letting $n\to\infty$ in \eqref{eq4.9} and using
dominated convergence once again we obtain the desired representation.
\end{proof}

\section{Stability results for solutions of PDEs with Dirichlet exterior condition}

Throughout the section, $L^n$ and $L$ are L\'evy operators with Fourier symbols $\psi_n$ and $\psi$, respectively.
We let $X^{0,n}$, $X^0$  denote associated L\'evy processes,  and  $P^n_x$, $P_x$
the corresponding  canonical laws  related to $L^n$ and $L$, respectively. 

Let us recall the assumptions (H1)--(H3) and introduce (H4) and (H5).
\begin{enumerate}
\item[(H1)] $\lim_{n\rightarrow\infty}\psi_n(\xi)=\psi(\xi)$ for $\xi\in\BR^d$.
\item[(H2)] 
$f\in C(\bar D\times \BR)$ and $g\in C(\BR^d)$;
\item[(H3)] $\big(f(x,y_1)-f(x,y_2)\big)(y_1-y_2)\le0$  for any $x\in D$ and $y_1,y_2\in\BR$;
\item[(H4$_x$)] $\int_{D^c} |g(z)|\,P_D(x,dz)<\infty$;
\item[(H5$_x$)] for every $n\ge 1$, $\int_{D^c} |g(z)|\,P^n_D(x,dz)<\infty$.
\end{enumerate}

\subsection{Pointwise stability}

As for the limiting process, we will need the assumption
\begin{equation}
\label{eq5.1}
\BP_x(\tau_{\bar D}=\tau_D)=1.
\end{equation}
Proposition \ref{prop6.1} and Example \ref{ex6.2} in the Appendix suggest that assumption \eqref{eq5.1}
is  close to being necessary in our main stability result proved below. 
Some sufficient conditions for \eqref{eq5.1} will be given in Section \ref{sec4}. 

We denote by $\VV$ the class of all  nondecreasing convex functions
$\varphi:[0,\infty)\rightarrow[0,\infty)$ such that
\[
\varphi(0)=0\quad\text{and}\quad
\lim_{r\rightarrow\infty}\frac{\varphi(r)}{r}=\infty.
\]

\begin{theorem}
\label{th3.4}
Let $u$ be a probabilistic solution of \eqref{eq1.1} with exceptional set $N$ and for $n\ge1$ let $u_n$ be a probabilistic solution of 
\eqref{eq1.1} with $L$ replaced by $L^n$ with exceptional set $N_n$. Set $N_\infty=\bigcup_{n} N_n$. Suppose that \mbox{\rm(H1)} is satisfied and
\begin{enumerate}[\rm(i)]
\item condition    \mbox{\rm(\ref{eq5.1})} is satisfied for every $x\in D\setminus N$;
\item $f,g$ satisfy  \mbox{\rm(H2)--(H3)};
\item  \mbox{\rm(H4$_x$)} holds for $x\in D\setminus N$, \mbox{\rm (H5$_x$)} holds for $x\in D\setminus N_\infty$;
\item there exists $\varphi\in \VV$ such that
\[
\sup_{n\ge1} \int_{D^c}\varphi(|g(y)|)P^n_D(x,dy)<\infty,\quad x\in D\setminus N_\infty.
\]
\end{enumerate}
Then $u_n(x)\rightarrow u(x)$ for every $x\in D\setminus (N\cup N_\infty)$.
\end{theorem}
\begin{proof}
Fix $x\in D\setminus (N\cup N_\infty)$.  By Theorem \ref{th4.2},  it suffices to show that
\begin{equation}
\label{eq5.3}
Y^{x,n}_0\arrowdl Y^{x}_0\quad\mbox{in }\BR,
\end{equation}
where $Y^{x}$ (resp. $Y^{x,n}$) is the first component of the solution of BSDE${}^{\tau^x_D}(\zeta^x,f,X^x,\BG,\BP)$ with  $\zeta^x=g(X^x_{\tau^x_D})$ 
(resp. BSDE${}^{\tau^{x,n}_D}(\zeta^{x,n},f,X^{x,n},\BG^n,\BP)$ with $\zeta^{x,n}=g(X^{x,n}_{\tau^{x,n}_D})$).
To show (\ref{eq5.3}) we will use Theorem \ref{th2.4} with $X=X^x$,   $X^n=X^{x,n}$, $\tau=\tau^x_D$, 
$\tau_n=\tau^{x,n}_D$ and $T>0$ chosen so  that $\BP(\tau^x_D=T)=0$ and $\BP(\Delta X^x_T=0)=1$ 
(see Remark \ref{rem2.8}). Hypothesis (H1) implies that $\lim_{n\rightarrow\infty}
\sup_{s\le t\le q} |e^{-(t-s)\psi_n(\xi)}-e^{-(t-s)\psi(\xi)}|=0$ for all $\xi\in\BR^d,q>0$, 
so by \cite[Corollary 2]{JS}, $X^{x,n}\arrowd X^x$ in $D([0,T];\BR^d)$ for $T>0$. 
By the Skorokhod representation theorem, we may assume that the convergence  holds in probability $\BP$. 
Since the processes $X^{x,n}$ have independent increments, \cite[Proposition 2]{CMS} implies that in fact  condition (\ref{eq2.2})  
is satisfied for every $T>0$. 
From Propositions \ref{prop6.1}, \ref{prop4.3} and
(\ref{eq5.1}) we have  $\tau^{x,n}_D\arrowp\tau^x_D$ for every  $x\in D\setminus N$.
Clearly $\BP(\tau^x_D>0)=1$.
Consequently, by Proposition \ref{prop4.2},  the first three conditions in 
\eqref{eq2.9} are  satisfied. 
Furthermore, by  Propositions \ref{prop6.1}, \ref{prop4.3} again, also the fourth condition in (\ref{eq2.9}) is satisfied. 
Next, \eqref{eq2.15} holds thanks to (iv) while  (a), (b) of Theorem \ref{th2.4} hold thanks to (iii).
Consequently, the assertion follows from  Theorem \ref{th2.4}.
\end{proof}

\begin{remark}
\label{rem5.2}
There is a qualitative difference between condition \textup{(i)} and conditions
\textup{(ii)--(iv)} of Theorem \ref{th3.4}. The latter concern only the
regularity, monotonicity, growth, and integrability properties of the data
$f$ and $g$ and are usually straightforward to verify. For example, if
\textup{(H2)}, \textup{(H3)} hold and
\[
 g \in C_b(\BR^d),
\]
then conditions \textup{(ii)--(iv)} are satisfied. By contrast, condition \textup{(i)} is of a genuinely different nature. It
concerns the interplay between the geometry of the boundary $\partial D$ and
the path behaviour of the limiting process $X$, in particular the way in
which $X$ may approach, hit or jump across the boundary.
\end{remark}

We say that $D$ is of class  $C^0$ if for any   $z \in$ $\partial D$ 
there exist  an open set  $O_z\subset\mathbb R^d$,  an isometry  $T_z:\BR^d\to\BR^d$
and a continuous function  $F_z:\mathbb R^{d-1}\to\mathbb R$ such that
\[
z\in O_z,\quad T_z\left(D \cap O_z\right)=T_z\left(O_z\right) \cap\left\{\left(x^{\prime}, x_d\right): x_d<F_z\left(x^{\prime}\right)\right\} .
\]

In Section \ref{sec4} we shall show that if $D$ is a bounded open set of class $C^0$ and $X$ satisfies (ACP), then  condition \eqref{eq5.1} is satisfied. 
This 
together with Theorem \ref{th3.4}  gives the following easily applicable result.

\begin{corollary}
\label{cor5.3}
Let $u$ be a probabilistic solution of \eqref{eq1.1}  and $u_n$ be a probabilistic solution of  \eqref{eq1.1} with $L$ replaced by 
$L^n$ (both with empty exceptional sets). Assume that \mbox{\rm(H1)} holds and
\begin{enumerate}[\rm(a)]
\item $D$ is of class $C^0$ and \mbox{\rm(ACP)} is satisfied;
\item $f$ satisfies  \mbox{\rm(H2)--(H3)} and $g\in C_b(\BR^d)$.
\end{enumerate}
Then $u_n(x)\rightarrow u(x)$ for every $x\in D$.
\end{corollary}
\begin{proof}
It follows from Remark \ref{rem5.2}, Proposition \ref{prop6.5} and Theorem \ref{th3.4}.
\end{proof}

\subsection{$L^p$ stability}

\begin{theorem}
\label{th3.10}
Let $p\ge 1$. Let $u$ be a probabilistic solution to \eqref{eq1.1}  and $u_n$ be a probabilistic solution to 
\eqref{eq1.1} with $L$ replaced by $L^n$. 
Suppose that \mbox{\rm(H1)} holds and:
\begin{enumerate}[\rm(a)]
\item there exists an $m$-inessential set $N\subset\BR^d$ such that     \mbox{\rm(\ref{eq5.1})} is satisfied for $x\in D\setminus N$;
\item $f$ satisfies  \mbox{\rm(H3)} and $g\in C(\BR^d)$;
\item $P_D[|g|^p],P_D^n[|g|^p]\in L^1(D)$ and $f(\cdot,y)\in L^1(D)$ for every $y\in\BR$;
\item $f-f(\cdot,0)\in C(\bar D\times\BR)$, $f(\cdot,0)\in L^p(D)$;
\item there exists $\varphi\in \mathcal V$ such that
$\sup_{n\ge 1} \int_D\int_{D^c}\varphi(|g(y)|^p)P^n_D(x,dy)\,dx<\infty$.
\end{enumerate}
 Then
$u_n\rightarrow u$ in $L^p(D)$.  
\end{theorem}
\begin{proof}
We denote by $N_n$ the exceptional set for $u_n$ and by  $N_0$ the exceptional set for $u$.
Since the union of $m$-inessential sets is an $m$-inessential set,
we may assume that $N_0=N$. 
By Theorem \ref{th4.2}, for any $x\in D\setminus N$, $u(X)$ is the first component of the solution of 
$\operatorname{BSDE}^{\tau_D}(g(X_{\tau_D}),f,X,\mathbb F,\BP_x)$,
and for any $x\in D\setminus N_n$, $u_n(X)$ is the first component of the solution of $\operatorname{BSDE}^{\tau_D}(g(X_{\tau_D}),f,X,\BF,\BP^n_x)$.
Next, notice that   thanks to (c), we have $u,u_n\in L^p(D)$.
Indeed, by Theorem \ref{th2.2},
\[
|u(x)|\le P_D|g|(x)+\int_D|f(y,0)|\,G_D(x,dy),\quad x\in D\setminus N,
\]
and by Fubini's theorem and Proposition \ref{prop4.2},
\[
\int_D\Big(\int_D|f(y,0)|\,G_D(x,dy)\Big)\,dx\le C\|f(\cdot,0)\|_{L^1(D)},
\]
which when combined with $P_D[|g|^p]\in L^1(D)$ shows that $u\in L^1(D)$ if $p=1$.
If $p>1$, then
by It\^o's formula (see, e.g., \cite[Theorem IV.66]{P}),
\begin{equation}
\label{eq5.4}
|u(x)|^p\le \BE_x|g(X_{\tau_D})|^p
+p\BE_x\int^{\tau_D}_0\operatorname{sgn}(u(X_t))|u(X_t)|^{p-1}f(X_t,u(X_t))\,dt.
\end{equation}
By using (H3) and  the elementary inequality 
\begin{equation}
\label{eq5.5}
ab\le \frac{(\varepsilon a)^{p/(p-1)}}{p/(p-1)}
+\frac{(\varepsilon^{-1}b)^{p}}{p},\quad a,b,\varepsilon>0,
\end{equation}
we show that the right-most term in \eqref{eq5.4} does not exceed
\begin{align}
\label{eq5.6}
&p\BE_x\int^{\tau_D}_0|u(X_t)|^{p-1}|f(X_t,0)|\,dt\nonumber\\
&\qquad\le  \BE_x\int^{\tau_D}_0
\big\{(p-1)\varepsilon^{p/(p-1)}|u(X_t)|^{p}
+\varepsilon^{-p}|f(X_t,0)|^p\big\}\,dt \nonumber\\
&\qquad=\int_D\big\{(p-1)\varepsilon^{p/(p-1)}|u(y)|^{p}
+\varepsilon^{-p}|f(y,0)|^p\big\}\,G_D(x,dy).
\end{align}
Taking now $\varepsilon=(2(p-1)C)^{-(p-1)/p}$, where $C$ is the constant from Proposition \ref{prop4.2}
applied to the dual operator $L^*$ and with  $q=1$, we have  $(p-1)\varepsilon^{p/(p-1)}=1/(2C)$. Therefore from 
Proposition \ref{prop4.2} and (\ref{eq5.4}) it follows that
\[
\|u\|^p_{L^p(D)}\le 2\|P_D|g|^p\|_{L^1(D)}+Cc(p,C)\|f(\cdot,0)\|^p_{L^p(D)},\quad 
c(p,C)=2(2(p-1)C)^{p-1}.
\]
In the same way we show that the above inequality is satisfied with $u$ replaced by $u_n$. To show the convergence of $u_n$ to $u$, we 
choose a sequence  $(\gamma_k)\subset C(\bar D)$ such that  $\gamma_k\rightarrow f(\cdot,0)$ in $L^p(D)$. 
For  $k\ge1$ define $f_k:D\times\BR\rightarrow\BR$ and $g_k:\mathbb R^d\to\mathbb R$ by 
\[
f_k(x,y)=f(x,y)-f(x,0)+\gamma_k(x),\quad g_k(x):=(g(x)\wedge k)\vee (-k). 
\]
By (c) and Remark \ref{rem3.2}, there exist  $m$-inessential sets $N',N'_n\subset D$ (with respect to $L$ and $L^n$, respectively)
such that (H4$_x$) holds for $x\in D\setminus N'$ and $\int_{D^c} |g(z)|\,P^n_D(x,dz)<\infty$ for $x\in D\setminus N'_n$. 
Clearly, we may assume that $N=N'$ and $N'_n=N_n$.
Notice that each  function $f_k$ satisfies (H2) and (H3). We denote by $u^k$   the solution of (\ref{eq1.1}) with $(f,g)$ replaced by $(f_k,g_k)$
and with an exceptional set $N$, 
while by $u^k_n$ we denote the solution of (\ref{eq1.1}) with $(f,g,L)$ replaced by $(f_k,g_k,L^n)$ and with an exceptional set $N_{k,n}$
(the existence follows from Theorem \ref{th4.2}).
By Theorem \ref{th3.4}, for a.e. $x\in D$,
\begin{align}
\label{eq5.7}
u^k_n(x)\rightarrow u^k(x)
\end{align}
as $n\rightarrow\infty$. Since $g_k,\gamma_k$ are bounded, it follows from Proposition \ref{prop2.4} and 
Proposition \ref{prop4.2} that $u^k$ is bounded and $u^k_n$
are bounded uniformly in $n\ge1$. Therefore (\ref{eq5.7}) implies that for each $k\ge1$,
\begin{equation}
\label{eq5.8}
\lim_{n\rightarrow\infty}\|u^k_n-u^k\|_{L^p(D)}=0
\end{equation}
for $p\ge1$. 
The rest of the proof is divided into two cases: $p=1$ and $p>1$. Let $p=1$. We shall show that
\begin{equation}
\label{eq5.9}
\lim_{k\rightarrow\infty}\|u^k-u\|_{L^1(D)}=0.
\end{equation}
By Proposition \ref{prop2.4},
\begin{align}
\label{eq5.10}
|u(x)-u^k(x)|\le c\Big(P_D|g-g_k|(x)+\int_D|f(y,0)-\gamma_k(y)|\,G_D(x,dy)\Big).
\end{align}
By this and  Proposition \ref{prop4.2},  
\[
\|u-u^k\|_{L^1(D)}\le C(\|P_D|g-g_k|\|_{L^1(D)}+\|f(\cdot,0)-\gamma_k\|_{L^1(D)}).
\]
This proves (\ref{eq5.9}). We next show that
\begin{equation}
\label{eq5.11}
\lim_{k\rightarrow\infty}\limsup_{n\rightarrow\infty}
\|u^k_n-u_n\|_{L^1(D)}=0.
\end{equation}
In much the same way, by using  Proposition \ref{prop2.4} and Proposition \ref{prop4.2}, we get 
\begin{align}
\label{eq5.12}
|u_n(x)-u_n^k(x)|\le c\Big(P^n_D|g-g_k|(x)
+\int_D|f(y,0)-\gamma_k(y)|\,G^n_D(x,dy)\Big),
\end{align}
and hence
\begin{align*}
\|u^k_n-u_n\|_{L^1(D)} 
\le C(\|P^n_D|g-g_k|\|_{L^1(D)}+ \|f(\cdot,0)-\gamma_k\|_{L^1(D)})
\end{align*}
for all  $n$, which  combined with (v) proves  (\ref{eq5.11}). From  (\ref{eq5.8})--(\ref{eq5.11}) we get the desired convergence in case  $p=1$. 

Suppose now that  $p>1$. Then by It\^o's formula (see, e.g., \cite[Theorem IV.66]{P}),
\begin{align*}
|u(x)-u^k(x)|^p&\le \BE_x|g(X_{\tau_D})-g_k(X_{\tau_D})|^p\\
&\quad+p\BE_x\int^{\tau_D}_0
\mbox{sgn}(u(X_t)-u^k(X_t))
|u(X_t)-u^k(X_t)|^{p-1}\\
&\qquad\qquad\qquad\qquad\qquad
\times(f(X_t,u(X_t))-f_k(X_t,u^k(X_t)))\,dt\\
&\le \BE_x|g(X_{\tau_D})-g_k(X_{\tau_D})|^p\\
&\quad+p\BE_x\int^{\tau_D}_0
|u(X_t)-u^k(X_t)|^{p-1}|f(X_t,u(X_t))-f_k(X_t,u(X_t))|\,dt.
\end{align*}
By using (\ref{eq5.5})  with $\varepsilon$ defined right after (\ref{eq5.6})  we get  
\begin{align*}
&|u(x)-u^k(x)|^p\le\BE_x|g(X_{\tau_D})-g_k(X_{\tau_D})|^p\\
&\quad\quad+  \BE_x\int^{\tau_D}_0
\Big\{\frac{1}{2C}|u(X_t)-u^k(X_t)|^{p}
+\frac{c(p,C)}{2}|f(X_t,u(X_t))-f_k(X_t,u(X_t))|^p\Big\}\,dt\\
&\quad=P_D|g-g_k|^p(x)+\int_D\Big\{\frac{1}{2C}|u(y)-u^k(y)|^{p}
+\frac{c(p,C)}{2}|f(y,0)-\gamma_k(y)|^p\Big\}\,G_D(x,dy).
\end{align*}
Integrating both sides of the above inequality with respect to $x$ we get
\begin{equation}
\label{eq5.13}
\|u-u^k\|^{p}_{L^p(D)}\le C c(p,C)\left(\|P_D|g-g_k|^p\|_{L^1(D)}+\|f(\cdot,0)-\gamma_k\|^p_{L^p(D)}\right).
\end{equation}
In the same way we get
\begin{equation}
\label{eq5.14}
\|u_n-u^k_n\|_{L^p(D)}^p \le  C c(p,C)\left(\|P^n_D|g-g_k|^p\|_{L^1(D)}+\|f(\cdot,0)-\gamma_k\|^p_{L^p(D)}\right).
\end{equation}
Combining  (\ref{eq5.8}) with  (\ref{eq5.13}), (\ref{eq5.14}) proves the theorem in case $p>1$.
\end{proof}

\begin{corollary}
\label{cor5.5}
Let $p\ge 1$. Let $u$ be a probabilistic solution of \eqref{eq1.1} and $u_n$ be a probabilistic solution of \eqref{eq1.1} with $L$ replaced by $L^n$. 
Suppose that \mbox{\rm(H1)} holds and
\begin{enumerate}[\rm(a)]
\item $D$ is of class $C^0$ and \mbox{\rm(ACP)} is satisfied;
\item $f$ satisfies  \mbox{\rm(H3)} and $g\in C_b(\mathbb R^d)$;
\item $f(\cdot,y)\in L^1(D)$ for every $y\in\BR$;
\item $f-f(\cdot,0)\in C(\bar D\times\mathbb R)$, $f(\cdot,0)\in L^p(D)$.
\end{enumerate}
 Then
$u_n\rightarrow u$ in $L^p(D)$.  
\end{corollary}

\subsection{Pointwise stability under additional Green function control} 

Below  we shall show that with the same assumptions  as in the case of $L^p$-stability studied  in the 
previous subsection and under additional control on the Green functions related to the sequence of operators $(L^n)$
one can show pointwise convergence.

In this subsection, we assume additionally that there exists $q>1$ such that 
\begin{equation}
\label{eq5.15}
\sup_{x\in D}\big( \|G_D(x,\cdot)\|_{L^q(D)}+\sup_{n\ge 1} \|G^{n}_D(x,\cdot)\|_{L^q(D)}\big)<\infty.
\end{equation}
By  (\ref{eq5.15}) and H\"older's inequality,  for any $p\ge q'$ we have
\[
\BE^n_x\int^{\tau_D}_0|f(X_t,0)|\,dt=\int_DG^{n}_D(x,y)|f(y,0)|\,dy\le c\|f(\cdot,0)\|_{L^p(D)}.
\]

\begin{proposition}
Let $u$ be a probabilistic solution of \eqref{eq1.1} with an exceptional  set $N$ and $u_n$ be a probabilistic solution of 
\eqref{eq1.1} with $L$ replaced by $L^n$ with an exceptional set $N_n$. Set $N_\infty=\bigcup_{n} N_n$.
Assume that \mbox{\rm(H1)} and \eqref{eq5.15} are satisfied  and
\begin{enumerate}[\rm(a)]
\item $D$ is of class $C^0$ and \mbox{\rm(ACP)} is satisfied;
\item $f$ satisfies  \mbox{\rm(H3)} and $g\in C_b(\BR^d)$;
\item $f(\cdot,y)\in L^{q'}(D)$ for every
$y\in\BR$;
\item $f-f(\cdot,0)\in C(\bar D\times\mathbb R)$, and $f(\cdot,0)\in L^p(D)$ for some $p\ge q'$.
\end{enumerate}
Then $u_n(x)\rightarrow u(x)$ for every $x\in D\setminus (N\cup N_\infty)$.
\end{proposition}
\begin{proof}
For $k\ge1$ we define  $f_k$ as in the proof of Theorem \ref{th3.10}  
and as in the proof of that theorem we define the functions  $u^k$, $u^k_n$. 
By Remark \ref{rem5.2} and Propositions \ref{prop4.3}, \ref{prop6.5}, condition (i) of Theorem \ref{th3.4} holds.
Furthermore,  applying H\"older's inequality
and using (\ref{eq5.10}), (\ref{eq5.15}) shows that  for any $x\in D\setminus N$ and $k\ge \|g\|_\infty$,
\begin{equation}
\label{eq5.16}
|u(x)-u^k(x)|\le c\|f(\cdot,0)-\gamma_k\|_{L^p(D)}.
\end{equation}
Similarly,  using  (\ref{eq5.12}), (\ref{eq5.15}) we obtain that for every $x\in D\setminus N_\infty$, 
\begin{equation}
\label{eq5.17}
|u_n(x)-u^k_n(x)|\le c\|f(\cdot,0)-\gamma_k\|_{L^p(D)}.
\end{equation}
In our case by Theorem \ref{th3.4}, the convergence \eqref{eq5.7} holds for every $x\in D\setminus(N\cup N_\infty)$.
Combining (\ref{eq5.7}) with (\ref{eq5.16}) and (\ref{eq5.17}) shows the desired convergence.
\end{proof}

\section{Geometric condition (\ref{eq5.1})}
\label{sec4}

The problem of checking condition (\ref{eq5.1}) for processes associated with general operators of the form \eqref{eq1.2} is  rather subtle.
Our aim is to show 	that  (\ref{eq5.1}) is satisfied for L\'evy processes satisfying condition (ACP) in the case when $D$ 
is a bounded open set of class $C^0$.

\begin{lemma}
\label{lem6.3}
Let $D\subset\BR^d$ be an open bounded set such that $m(\partial D)=0$ and
$X$ be a L\'evy process in $\BR^d$ satisfying condition \mbox{\rm(ACP)}. Then, for every $x\in\BR^d\setminus \partial D$,
\begin{equation}
\label{eq6.2}
\BP_x(\exists s>0:\ X_{s-}\in \BR^d\setminus \partial D,\ X_s\in \partial D)=0.
\end{equation}
\end{lemma}
\begin{proof}
Set $A=\partial D$ and define  $h(y)=\fch_{A^c}(y)\,\nu(A-y)$, $y\in\BR^d$,
\[
N_A=\sum_{s>0}\fch_{\{X_{s-}\notin A,\ X_s\in A\}}.
\]
By using the L\'evy system (\ref{eq6.1}) for $X$ and (\ref{eq3.5}) we obtain 
\begin{equation}
\label{eq6.3}
\BE_xN_A
=\BE_x\int_0^\infty \mathbf{1}_{\{X_t\notin A\}}\,\nu(A-X_t)\,dt
=\BE_x\int_0^\infty h(X_t)\,dt.
\end{equation}
By Tonelli's theorem,
\[
\int_{\BR^d} h(y)\,dy=\int_{\BR^d}
\Big(\int_{\BR^d}\fch_{A^c}(y)\fch_{A}(y+z)\,dy\Big)\nu(dz).
\]
For each fixed $z\in\BR^d\setminus\{0\}$ we have
\[
\int_{\BR^d}\fch_{A^c}(y)\fch_{A}(y+z)\,dy\le m(A-z)=0.
\]
Hence  $h=0$ a.e. in $\BR^d$. Since (ACP) is satisfied, it follows from this and (\ref{eq6.3}) that $\BE_x N_A=0$, which proves \eqref{eq6.2}. 
\end{proof}

\begin{lemma}
\label{lem6.4}
Let $X$ be a L\'evy process in $\BR^d$ satisfying condition \mbox{\rm(ACP)} and 
$U$ be an  open subset of $\BR^d$ of the  form
\[	
U=T(V)\cap B(z,r),
\]
where $T:\BR^d\to \BR^d$ is an isometry, $V$ is a hypograph of the form 
$V=\{(x,y): y<f(x)\}$   for  some $f\in C(\BR^{d-1})$
and  $B(z,r)=\{x\in\BR^d:|x-z|<r\}$ for some $z\in\BR^d$, $r>0$.
Then 
\[
\BP(\tau^x_U=\tau^x_{\bar U})=1,\quad x\in U,
\]
where $\tau^x_U=\inf\{t>0:X^x_t\notin U\}$, 
$\tau^x_{\bar U}=\inf\{t>0:X^x_t\notin\bar U\}$.
\end{lemma}
\begin{proof}
First observe that for any open sets $V, W$ such that $\tau^x_V=\tau^x_{\bar V}$ for $x\in V$ and $\tau^x_W=\tau^x_{\bar W}$ for $x\in W$ we have 
\begin{equation}
\label{eq6.4}
\tau^x_{V\cap W}=\tau^x_{\overline{V\cap W}},\quad x\in V\cap W.
\end{equation}
Indeed, $\overline{V\cap W}\subset \bar V\cap\bar W$, which implies that
\[
\tau^x_{V\cap W}\le \tau^x_{\overline{V\cap W}}\le
\tau^x_{\bar V\cap\bar W}=\tau^x_{\bar V}\wedge \tau^x_{\bar W}
=\tau^x_V\wedge \tau^x_W=\tau^x_{V\cap W},\quad x\in V\cap W.
\]
Since for any L\'evy process $X$ satisfying (ACP) and any isometry $T$, the process $Y=T(X)$ is again a L\'evy process satisfying (ACP), 
and since each open ball is the intersection of a finite number of rotated open hypographs of continuous functions, it follows from (\ref{eq6.4}) 
that we are reduced to proving that $\tau^x_U=\tau^x_{\bar U}$ for every $x\in U$, where $U$ is an open hypograph of a continuous function 
$f:\BR^{d-1}\to \BR$. Let $U$ be of this form.
Let $\vec  r=(0,\dots,0,r)\in \BR^d$. Note that 
\begin{equation}
\label{eq6.5}
\tau^x_{U}=\tau^{x-\vec r}_{U-\vec r},\qquad \tau^x_{\bar U}
=\tau^{x-\vec r}_{\bar U-\vec r}\,,\quad x\in U.
\end{equation}
From the proof of \cite[Proposition (2.5)]{D} it follows that  there exists $r>0$ such that  $\tau^{y}_{U-\vec r}
=\tau^{y}_{\bar U-\vec r}$ for a.e. $y\in U-\vec r$.
This together with (\ref{eq6.5})  yields  
$\tau^x_{U}= \tau^x_{\bar U}$ for a.e. $x\in U$, so the functions $w,\bar w$  defined by $w(x)=\BE e^{-\tau^x_{U}}$, $\bar w(x)=\BE e^{-\tau^x_{\bar U}}$ 
are equal a.e. in $U$. The functions  $w,\bar w$ are $1$-excessive with respect to  the semigroup $(P^U_t)$  (see, e.g., \cite[Proposition II.(2.8)]{BG}).
Since (ACP)  is satisfied, in fact
$w(x)=\bar w(x)$ for every  $x\in U$ by \cite[Proposition VI.(1.3)]{BG}. Thus $\BE(e^{-\tau^x_{U}}-e^{-\tau^x_{\bar U}})=0$ for $x\in U$, 
which proves the lemma since $\tau^x_{\bar U}\ge\tau^x_U\,$a.s.
\end{proof}

\begin{proposition}
\label{prop6.5}
Let $D$ be a bounded open subset of $\BR^d$ of class $C^0$.
If  $X$ is a L\'evy process satisfying condition \mbox{\rm(ACP)},  then \eqref{eq5.1} is satisfied.
\end{proposition}
\begin{proof}
We  let $\tau=\tau_{D}$ and $\bar\tau=\tau_{\bar D}$. Fix $x\in D$. Since  $\BP_x(\tau\le \bar\tau)=1$, it suffices to prove that 
\begin{equation}
\label{eq6.6}
\mathbb P_x(\tau<\bar\tau)=0.
\end{equation}
By Lemma \ref{lem6.3},
\begin{equation}
\label{eq6.7}
\BP_x(\exists{s>0}:X_{s-}\in D,\, X_s \in \partial D)=0.
\end{equation}
Since  $\tau$ is the first exit time from $D$, $X_\tau\notin D$. Hence $X_\tau\in\partial D$ on the set $\{\tau<\bar\tau\}$. From this,  \eqref{eq6.7} and 
Proposition \ref{prop4.3} we get 
\begin{equation}
\label{eq6.8}
\BP_x(\tau<\bar\tau)=\BP_x(A),\quad\mbox{where}\quad A=\{\tau<\bar\tau,\,  X_{\tau-}=X_\tau,\, X_\tau\in \partial D\}.
\end{equation}
Let $\RR$ denote the family of sets $T(V)$, where
$T:\BR^d\to\BR^d$ is an isometry and
$V=\{(x',x_d):x_d<F(x')\}$ for some
$F\in C(\BR^{d-1})$.
For every $y\in \partial D$ there exist $r_y>0$ and $V_y\in\RR$ such that
\[
D\cap B(y,3r_y)=V_y\cap B(y,3r_y).
\]	
Since $\partial D$ is compact, there exists a  finite set  
$\{y_1,\dots,y_N\}\subset \partial D$  such that the  balls $B(y_k,r_{y_k})$ cover $\partial D$. To shorten notation, set 
\[
b_k=B(y_k,r_{y_k}),\quad B_k= B(y_k,2r_{y_k}),\quad V_k= V_{y_k},\quad  U_k= V_k\cap B(y_k,3r_{y_k}).
\]
For each $k\in\{1,\dots,N\}$  define recursively the entrance times $\sigma_{k,n}$  and exit times $\eta_{k,n}$ by 
\[
\sigma_{k,1}=\inf\{t>0:X_t\in D\cap b_k\},\qquad
\eta_{k,1}=\inf\{t>\sigma_{k,1}:X_t\notin B_k\},
\]
and, for $n\ge1$,
\[
\sigma_{k,n+1}=\inf\{t>\eta_{k,n}:X_t\in D\cap b_k\},\qquad
\eta_{k,n+1}=\inf\{t>\sigma_{k,n+1}:X_t\notin B_k\}.
\]
On the set $A$  we have $X_\tau\in \partial D$, so $X_\tau\in b_k$ for some $k\ge1$. Consequently,
\[	
A\subset \bigcup_{k=1}^N\bigcup_{n=1}^{\infty}\{X_\tau\in b_k,\,\sigma_{k,n}<\tau<\eta_{k,n}\}.
\]
Observe now that 
\[
A\cap\{X_\tau\in b_k,\ \sigma_{k,n}<\tau<\eta_{k,n}\}\subset
\{\tau_{U_k}\circ\theta_{\sigma_{k,n}}<
\tau_{\bar U_k}\circ\theta_{\sigma_{k,n}}\},
\]
and
\[
\{\sigma_{k,n}<\tau\}
\subset
\{X_{\sigma_{k,n}}\in U_k\}.
\]
Therefore, by Lemma \ref{lem6.4} and the
strong Markov property at $\sigma_{k,n}$,
\[
\BP_x(A\cap\{X_\tau\in b_k,\,\sigma_{k,n}<\tau<\eta_{k,n}\})
\le \BE_x\Big(\fch_{\{\sigma_{k,n}<\tau\}}\BP_{X_{\sigma_{k,n}}}\left(\tau_{U_k}<\tau_{\bar U_k}\right)\Big)
=0.
\]
Summing over $k,n$ we get  $\BP_x(A)=0$, which together with \eqref{eq6.8} proves \eqref{eq6.6}.
\end{proof}

\section{Applications}
\label{sec7}

In this section, we first give examples of operators whose symbols $\psi_n,\psi$, i.e. characteristic exponents 
of the corresponding L\'evy processes, satisfy (H1). In the second part, as an application of our results 
for linear Dirichlet problems, we show a result on the Mosco convergence of parts on $D$ of the  Dirichlet forms associated with $\psi_n,\psi$.
 
\subsection{Convergence of solutions of semilinear equations}

In what follows,  we assume that
 $u$ is a probabilistic solution to \eqref{eq1.1}  and $u_n$ is a probabilistic solution to 
\eqref{eq1.1} with $L$ replaced by $L^n$.
Furthermore, we assume that $D$ is of class $C^0$,
$f,g$ satisfy  \mbox{\rm(H2)--(H3)} and  $g\in C_b(\mathbb R^d)$. Accordingly we may and will consider in this section 
only canonical solutions to \eqref{eq1.1} with $N=\emptyset$.

\begin{example}
\label{ex7.1}
(Fractional  Laplace operators). Let $L^n=-(-\Delta)^{\alpha_n/2}$, $L=-(-\Delta)^{\alpha/2}$ with $\alpha_n,\alpha\in(0,2]$. Then   
$\psi_n(\xi)=|\xi|^{\alpha_n}$, $\psi(\xi)=|\xi|^{\alpha}$.
If $\alpha_n\rightarrow\alpha$, then from Corollary \ref{cor5.3}  it follows that $u_n\rightarrow u$ pointwise in $D$. 
In the case $\alpha=2$ the operator associated with $\psi$ is the Laplace operator $\Delta$. 
\end{example}

\begin{example}
(Multifractal operators). Let $L^n=-\sum^d_{i=1}(-\frac{\partial^2}{\partial x_i^2})^{\alpha^i_n/2}$ 
with $\alpha^i_n\in(0,2)$, $i=1,\dots,d$.  Then  $\psi_n(\xi)=|\xi_1|^{\alpha^1_n}+\ldots+|\xi_d|^{\alpha^d_n}$.
The process with symbol $\psi_n$ has the form $X^{x,n}=(x_1+X^{1,n},\dots,x_d+X^{d,n})$, 
where $X^{1,n},\dots, X^{d,n}$ are one-dimensional, independent, symmetric  $\alpha^i_n$-stable
processes.   If $\alpha^i_n\rightarrow2$ for $i=1,\dots,d$, then  $\psi_n(\xi)\rightarrow\psi(\xi)=|\xi|^2$. 
By Corollary \ref{cor5.3},   $u_n\rightarrow u$, where $u$ is the solution of (\ref{eq1.1}) with $L=\Delta$.
\end{example}

\begin{example}
(Mixed local and  fractional  operators).
Let $L^n=b_n\cdot\nabla+a_n\Delta-(-\Delta)^{\alpha_n/2}-(-\Delta)^{\beta_n/2}$ with $b_n\in\BR^d$, $a_n>0$ and $\alpha_n,\beta_n\in(0,2)$. \\
(i) If $b_n\rightarrow b$, $a_n\rightarrow a$ and  $\alpha_n\rightarrow\alpha$, $\beta_n\rightarrow \beta$ for some $b\in\BR^d$, $a\ge0$ 
and $\alpha,\beta\in(0,2)$, then by  Corollary \ref{cor5.3},  $u_n\rightarrow u$ pointwise in $D$, where $u$ is the probabilistic solution of problem (\ref{eq1.1}) with 
$L=b\cdot\nabla +a\Delta-(-\Delta)^{\alpha/2}-(-\Delta)^{\beta/2}$. \\
(ii) If $b_n\rightarrow0$, $a_n\rightarrow1$, $\alpha_n\rightarrow2$, $\beta_n\rightarrow2$, then $u_n\rightarrow u$, where 
$u$ is the solution of (\ref{eq1.1}) with $L=3\Delta$.  
\end{example}

\begin{example}
(Relativistic fractional operators). Consider the relativistic fractional Laplacian
\[
L^n=-\big(-\Delta+(m_n)^{2/\alpha_n}\big)^{\alpha_n/2}+m_n,
\qquad
\alpha_n\in(0,2),\quad m_n>0,
\]
with its L\'evy--Khintchine symbol
\[
\psi_n(\xi)=((m_n)^{2/\alpha_n}+|\xi|^2)^{\alpha_n/2}-m_n,\quad \xi\in\BR^d.
\]
The L\'evy process with the characteristic exponent $\psi_n$ is called the relativistic $\alpha_n$-stable process. 
An explicit expression for its L\'evy measure  can be found in \cite{CKS}.
\\
(i) If $\alpha_n\rightarrow2$  and $m_n\rightarrow m>0$, then by Corollary \ref{cor5.3}, $u_n\rightarrow u$ pointwise in $D$, 
where $u$ is the solution of (\ref{eq1.1}) with $L=\Delta$.\\
(ii) If $\alpha_n\rightarrow\alpha$ for some $\alpha\in(0,2)$  and $m_n\rightarrow0$, then $u_n\rightarrow u$, where $u$ 
is the solution of (\ref{eq1.1}) with $L=-(-\Delta)^{\alpha/2}$.  
\end{example}

\begin{example}(Fractional logarithmic Schr\"odinger operators). 	
Let $\psi_n(\xi)=\ln(1+|\xi|^{\alpha_n})$  with $\alpha_n\in(0,2)$ and $\psi(\xi)=\ln(1+|\xi|^2)$. The L\'evy process $X^n$ 
with characteristic exponent $\psi_n$ is called a geometric $\alpha_n$-stable process. The process $X$ corresponding to $\psi$ 
is often called the variance gamma process. The operators with symbols $\psi_n,\psi$ are denoted by  $L^n=-\ln(1+(-\Delta)^{\alpha_n/2})$ 
and $L=-\ln(1-\Delta)$, respectively. $L^n$ is called the fractional logarithmic Schr\"odinger operator, while $L$  is the logarithmic Schr\"odinger operator.  
From the Chung--Fuchs criterion (\ref{eq4.4}) it follows that $X$ is transient for $d\ge3$. The exponent $\psi$ does not satisfy the Hartman--Wintner 
condition (\ref{eq1.9}). It is known, however,  that $X$ satisfies condition (ACT) (see \cite{SSV} or the recent preprint by K.~Tsuchida, \emph{A note on geometric $\alpha$-stable processes and the existence of ground states for associated Schr\"odinger operators} (arXiv:2603.11570, 2026)). If $\alpha_n\rightarrow2$, 
then clearly (H1) is satisfied, so the solutions $u_n$ of (\ref{eq1.1})  corresponding to $L^n$ converge pointwise in $D$  to the solution $u$ of (\ref{eq1.1}) 
corresponding to $L$. 
\end{example}

The linear Dirichlet problem involving the fractional logarithmic Schrödinger operator has been studied from the PDE perspective in \cite{F}.

\begin{example} 
\label{ex5.7}	For $n\ge1$ let $\lambda_n>0$ and $\nu_n$ be a symmetric probability measure on $\BR^d$ such that $\nu_n(\{0\})=0$. Set  
\[
\psi_n(\xi)=|\xi|^2+\lambda_n\int_{\BR^d}(1-\cos(\xi,y))\,\nu_n(dy).
\]
The L\'evy process with the characteristic exponent $\psi_n$ is the sum of a $d$-dimensional Wiener process  with covariance matrix 
$2I_d$  and an independent compound Poisson process with L\'evy measure $\lambda_n\nu_n$. The operator corresponding to $\psi_n$ has the form
\[
L^n\varphi(x)=\Delta\varphi(x)+\lambda_n\int_{\BR^d}(\varphi(x+y)-\varphi(x))\,\nu_n(dy),\quad \varphi\in C^2_c(\BR^d)
\]
(see, e.g., \cite[Example 31.8]{Sa}). If $\lambda_n\rightarrow1$ and the sequence $\{\nu_n\}$ converges weakly to the Dirac 
measure $\delta_{0}$, then $\psi_n(\xi)\rightarrow|\xi|^2$ for $\xi\in\BR^d$. Therefore  Corollary \ref{cor5.3}  applies to solutions $u_n$ corresponding to $L^n$ and $u$ corresponding to $\Delta$.
\end{example}

In the next examples we consider stable  operators, i.e. operators being generators of $\alpha$-stable L\'evy  processes. 
Regularity properties of solutions of the Dirichlet problem involving operators of such kind are studied in \cite{R-OS,R-OV} in the 
symmetric case and in \cite{DR-OSV} in the general case. For convenience of exposition, below we recall some  facts on stable processes and their generators.

By \cite[Theorem 14.3]{Sa}, if $X$ is an $\alpha$-stable, then $Q=0$ and its L\'evy measure is of the form 
\begin{equation}
\label{eq7.1}
\nu(B)=\int_{S^{d-1}}\mu(dy)\int^{\infty}_0\fch_{B}(ry)r^{-1-\alpha}\,dr,\quad B\in\BB(\BR^d)
\end{equation}
for some finite measure $\mu$ on $S^{d-1}=\partial B(0,1)$ called the spherical part of the measure $\nu$.  We know that the restriction of the 
generator $L$ of $X$ to $C^2_c(\BR^d)$ is given  by (\ref{eq1.2})  but  for $\alpha$-stable  process both $\psi$  and (\ref{eq1.2}) have  
simpler forms. We give them  below for symmetric measures $\mu$ and $\alpha\in(0,1)$ or $\alpha\in(1,2)$. For the general case 
we refer the reader to \cite[Remark 14.6, Theorem 14.7]{Sa} for a description of $\psi$ and to  \cite{DR-OSV} for a simplified form of (\ref{eq1.2}). 

If $\mu$ is symmetric, i.e. $\mu(B)=\mu(-B)$ for any Borel subset $B$ of $S^{d-1}$, then $\nu$ is also symmetric and by \cite[Exercise E 18.8]{Sa}, the 
characteristic exponent (\ref{eq1.4}) (with $Q=0$) of the $\alpha$-stable L\'evy process with the L\'evy measure (\ref{eq7.1}) has the form
\begin{equation}
\label{eq7.2}
\psi(\xi)
=c_{\alpha}\int_{S^{d-1}}|(\xi,y)|^{\alpha}\,\mu(dy),
\end{equation}
where 
\[
c_{\alpha}=\sqrt{\pi}\frac{\Gamma((2-\alpha)/2)}{2^{\alpha}\alpha\Gamma((1+\alpha)/2)}.
\]
Moreover, for $\alpha\in(0,1)$,
\begin{equation}
\label{eq7.3}
L\varphi(x)=\int_{S^{d-1}}\mu(d\xi)\int^{\infty}_0
(\varphi(x+r\xi)-\varphi(x))r^{-1-\alpha}\,dr,\quad \varphi\in C^2_c(\BR^d).
\end{equation}
Indeed, if $\alpha\in(0,1)$, then by \cite[Proposition 14.5]{Sa}, $\int_{\BR^d}|y|\,\fch_{\{|y|\le 1\}}\,\nu(dy)<\infty$. Therefore, by \cite[(14.7)]{Sa},
\[
\int_{\BR^d}y_i\fch_{\{|y|\le1\}}\,\nu(dy)
=\int_{S^{d-1}}\mu(d\xi)\int^1_0r\xi_ir^{-1-\alpha}\,dr
=\frac{1}{1-\alpha}\int_{S^{d-1}}\xi_i\mu(d\xi)=0,
\]
which together with (\ref{eq1.2}) gives (\ref{eq7.3}). 
If $\alpha\in(1,2)$, then  $\int_{\BR^d}|y|\,\fch_{\{|y|>1\}}\,\nu(dy)<\infty$  by \cite[Proposition 14.5]{Sa}. 
Using this and symmetry of $\mu$ we obtain  
\begin{equation}
\label{eq7.4}
L\varphi(x)=\frac12\int_{S^{d-1}}\mu(d\xi)\int^{\infty}_0
(\varphi(x+r\xi)+\varphi(x-r\xi)-2\varphi(x))
r^{-1-\alpha}\,dr
\end{equation}
for $\varphi\in C^2_c(\BR^d)$.

If furthermore the spherical part of $\nu$ is the uniform probability distribution on $S^{d-1}$, i.e. is equal to $\sigma/\omega_{d-1}$, 
where $\sigma$ is the surface measure on $S^{d-1}$ and $\omega_{d-1}=\sigma(S^{d-1})=2\pi^{d/2}/\Gamma(d/2)$, then by \cite[Exercise E 18.9]{Sa},
\begin{equation}
\label{eq7.5}
\psi (\xi)=C_{\alpha}|\xi|^{\alpha}\quad\mbox{with}\quad \frac{1}{C_{\alpha}}=\frac{2^{\alpha}\alpha\Gamma((\alpha+d)/2)}
{\Gamma(d/2)\Gamma((2-\alpha)/2)}.
\end{equation}
Let
\begin{equation}
\label{eq7.6}
c_{d,\alpha}=\frac{1}{C_{\alpha}\omega_{d-1}}=\frac{2^{\alpha-1}\alpha\Gamma((d+\alpha)/2)}{\pi^{d/2}\Gamma((2-\alpha)/2)}.
\end{equation}
From the above it follows that if the spherical part $\mu$  equals $c_{d,\alpha}\sigma$, then 
$\psi(\xi)=|\xi|^{\alpha}$, i.e. the corresponding operator is $-(-\Delta)^{\alpha/2}$ 
(see also \cite[Example 32.7]{Sa}). 

\begin{example}(Anisotropic fractional operators I). Let $\alpha\in(0,2)$ and for $n\ge1$ let $\mu_n$ be a nontrivial symmetric  
finite measure on $S^{d-1}$. Define  a L\'evy  measure $\nu_n$ by
\[
\nu_n(B)=\int_{S^{d-1}}\mu_n(dy)\int^{\infty}_0\fch_{B}(ry)r^{-1-\alpha}\,dr,\quad B\in\BB(\BR^d),
\]
and then $\psi_n$ by (\ref{eq7.2}) with $\mu$ replaced by $\mu_n$. The function $\psi_n$ is the characteristic exponent of the L\'evy process with 
the generating triplet $(0,\nu_n,0)$.  The stable operator $L^n$ with symbol $\psi_n$  is sometimes called the anisotropic fractional operator.
Suppose that $\{\mu_n\}$  converges weakly to $c_{d,\alpha}\sigma$. For instance, we may consider $\mu_n=h_n\sigma$ with even positive 
functions  $h_n$ such that  $h_n\rightarrow c_{d,\alpha}$ in $L^1(S^{d-1};\sigma)$.
Then (H1) is satisfied, so by Corollary \ref{cor5.3} the solutions $u_n$ of (\ref{eq1.1})  corresponding to $L^n$ converge pointwise in $D$  
to the solution $u$ of (\ref{eq1.1}) with $L=-(-\Delta)^{\alpha/2}$.
\end{example}

\begin{example}
(Anisotropic fractional operators II). 	
For $n\ge1$ let $\alpha_n\in(1,2)$ and $\mu_n=c_{d,\alpha_n}h_n\cdot\sigma$, where $c_{d,\alpha_n}$ is defined by (\ref{eq7.6}) and 
$h_n\in L^1(S^{d-1};\sigma)$ is even and positive. Assume that $\alpha_n\rightarrow 2$ and $h_n\rightarrow 1$ in $L^1(S^{d-1};\sigma)$ 
as $n\rightarrow\infty$. By (\ref{eq7.2}),
\[
\psi_n(\xi)=c_{\alpha_n}c_{d,\alpha_n}
\int_{S^{d-1}}|(\xi,y)|^{\alpha_n}\,h_n(y)\,\sigma(dy).
\]
The corresponding operator $L^n$ can be written in the form (\ref{eq7.4}) with $\mu$ replaced by $\mu_n$ and $\alpha$ 
replaced by $\alpha_n$. By the definitions of the constants $c_{\alpha_n}$, $c_{d,\alpha_n}$ and (\ref{eq7.5}), 
\[
c_{\alpha_n}\cdot c_{d,\alpha_n}
\int_{S^{d-1}}|(\xi,y)|^{\alpha_n}\,\sigma(dy)
=|\xi|^{\alpha_n}\rightarrow|\xi|^2.
\]
Since the sequence $\{c_{\alpha_n}\cdot c_{d,\alpha_n}\}$ converges to a finite constant, 
\[
c_{\alpha_n}c_{d,\alpha_n}
\int_{S^{d-1}}|(\xi,y)|^{\alpha_n}|1-h_n(y)|\,\sigma(dy)\le C_\xi
\|1-h_n\|_{L^1(S^{d-1};\sigma)}
\longrightarrow0.
\]
Thus $\psi_n(\xi)\rightarrow|\xi|^2$ for $\xi\in\BR^d$,  so Corollary \ref{cor5.3} applies to operators $L^n$ and $L=\Delta$. 
If $h_n=1$ for $n\ge1$, then the above result reduces to the one given in Example \ref{ex7.1}.
\end{example}

We next give an example of nonsymmetric operators for which (H1) is satisfied. We will use 
\cite[Theorem 8.7]{Sa}. This result is a general tool which allows expressing the pointwise convergence $\psi_n\rightarrow\psi$  
in terms of the generating triplets of the corresponding processes.

\begin{example}(Nonsymmetric stable-like operators). Let $K, K_n:\BR^d\rightarrow\BR$ be  nonnegative functions such that 
$1/C\le K_n(y)\le C$, $1/C\le K(y)\le C$ for some $C>0$ and  $K_n\rightarrow K$ a.e.  as $n\rightarrow\infty$. Let $\alpha \in (0,2)$, and 
\[
\nu_n(dy)=K_n(y)|y|^{-d-\alpha}\,dy,\qquad	\nu(dy)=K(y)|y|^{-d-\alpha}\,dy,
\]
and $\psi_n,\psi$ denote the characteristic exponents of L\'evy processes with generating triplets $(0,\nu_n,0)$ and $(0,\nu,0)$, respectively. Then
\[
I_{n,\varepsilon}(\xi):=\int_{\{|y|<\varepsilon\}}|(\xi,y)|^2\,\nu_n(dy)\le C|\xi|^2\int_{\{|y|<\varepsilon\}}|y|^{2-d-\alpha}\,dy=C'|\xi|^2\int^{\varepsilon}_0r^{1-\alpha}\,dr.
\]
Hence $\lim_{\varepsilon\rightarrow0^+}\limsup_{n\rightarrow\infty}I_{n,\varepsilon}(\xi)=0$ for $\xi\in\BR^d$.
Furthermore, one can easily check that $\int_{\BR^d}f(y)\,\nu_n(dy)\rightarrow\int_{\BR^d}f(y)\,\nu(dy)$ for every bounded continuous 
function $f:\BR^d\rightarrow\BR$ vanishing on a neighborhood of 0 and 
that $\int_{\BR^d}y(1/(1+|y|^2)-\fch_{\{|y|<1\}})\,\nu_n(dy)\rightarrow
\int_{\BR^d}y(1/(1+|y|^2)-\fch_{\{|y|<1\}})\,\nu(dy)$ as $n\rightarrow\infty$.  Therefore, by  \cite[Theorem 8.7]{Sa}, 
$\psi_n(\xi)\rightarrow\psi(\xi)$ for every $\xi\in\BR^d$, so Corollary \ref{cor5.3} applies.
\end{example}

\subsection{Mosco convergence of parts of forms}

Suppose that (H1) is satisfied. Throughout this subsection, we 
assume that $\psi_n$ and $\psi$ are real-valued.
For $n\ge1$ define the Dirichlet form $(\Psi^{(n)},H^{\psi_n})$ on $L^2(\BR^d)$ by (\ref{eq3.1}), (\ref{eq3.2}) with $\psi$ 
replaced by $\psi_n$.   In this section, as a byproduct of our results proved in Sections \ref{sec3}--\ref{sec4},  we show the convergence, 
in the sense of Mosco, of parts of these forms on $D$ to the part  of $(\Psi,H^{\psi})$ on $D$. 
In the convergence result, we assume that $D$ satisfies the assumptions of Proposition \ref{prop6.5} and $X$ satisfies condition (ACP). We also give an 
example showing that in general the convergence does not hold for an arbitrary bounded open set $D$.

Although our result concerns a special class of forms, it is convenient to start with the general definition of Mosco convergence. 

\begin{definition}
Let $(\EE,\FF)$, $(\EE^{(n)},\FF^{(n)})$, $n\ge1$,
be closed quadratic forms on $H:=L^2(E;m)$, where $E$ is a separable metric space and $m$ is a $\sigma$-finite measure on it. We let 
\[
\EE(u,u)=\infty \quad \text{if } u\notin \FF,\qquad \EE^{(n)}(u_n,u_n)=\infty \quad \text{if } u_n\notin \FF^{(n)}.
\]	
We say that $(\EE^{(n)},\FF^{(n)})$ converges to $(\EE,\FF)$ in the sense of Mosco if
\begin{enumerate}	
\item[(M1)] For every sequence $\{u_n\}\subset H$ such that
$u_n\rightharpoonup u$ weakly in $H$ we have
\[
\liminf_{n\to\infty} \EE^{(n)}(u_n,u_n)\ge \EE(u,u);
\]		
\item[(M2)] For every $u\in H$ there is a sequence $\{u_n\}\subset H$ such that
$u_n \to u$ strongly in $H$ and
\[
\limsup_{n\to\infty}\EE^{(n)}(u_n,u_n)\le \EE(u,u).
\]
\end{enumerate}
\end{definition}

Let $\{G_{\alpha}\}_{\alpha>0}$, $\{G^{(n)}_{\alpha}\}_{\alpha>0}$ denote the resolvents  of $(\EE,\FF)$ and  $(\EE^{(n)},\FF^{(n)})$, 
respectively (see \cite{M} or \cite[Section 1.3]{FOT}). 

\begin{proposition}
\label{prop7.11}
Let $(\EE,\FF)$, $(\EE^{(n)},\FF^{(n)})$, $n\ge1$, be closed quadratic forms on $H$. Then  $(\EE^{(n)},\FF^{(n)})$ converges to $(\EE,\FF)$ in the sense of 
Mosco if and only if for any $\alpha>0$ and $f\in H$, $G^{(n)}_{\alpha}f$ converges to $G_{\alpha}f$ in $H$.	
\end{proposition}
\begin{proof}
This is a special case of   \cite[Theorem 2.4.1]{M}; see also \cite[Proposition 2.1]{CS}.
\end{proof}
In what follows we confine ourselves to the case where $(\EE,\FF)=(\Psi, H^{\psi})$ and  $(\EE^{(n)},\FF^{(n)})=(\Psi^{(n)}, H^{\psi_n})$, $n\ge1$. 
Recall that $(\Psi, H^{\psi})$, $(\Psi^{(n)}, H^{\psi_n})$ are symmetric Dirichlet forms on $L^2(\BR^d)=L^2(\BR^d;m)$, where $m$ is the 
Lebesgue measure on $\BR^d$. By definition, the parts of these forms on $D$ are  $(\Psi, H^{\psi}_0(D))$ and  $(\Psi^{(n)}, H^{\psi_n}_0(D))$, respectively.

It is known (see \cite[Lemma 3.2]{CS}) that if (H1) is satisfied, then the forms $(\Psi^{(n)}, H^{\psi_n})$ converge as $n\rightarrow\infty$ to $(\Psi, H^{\psi})$ 
in the sense of Mosco. The question is  whether (H1) also implies the convergence of parts of these forms on $D$.  A positive answer is given in 
Theorem \ref{th7.12} below. For applications of this theorem  we refer the reader to \cite{CS}.

We denote by
$\{G^{D}_{\alpha}\}_{\alpha>0}$, $\{G^{D,n}_{\alpha}\}_{\alpha>0}$ the resolvents on $L^2(D):=L^2(D;m)$ of the forms
$(\Psi, H^{\psi}_0(D))$, $(\Psi^{(n)}, H^{\psi_n}_0(D))$, and for any $\alpha>0$ and bounded Borel $f:D\rightarrow\BR$ we set 
\[
R^D_{\alpha}f(x)=\BE_x\int^{\tau}_0e^{-\alpha t}f(X_t)\,dt,\qquad R^{D,n}_{\alpha}f(x)=\BE_x\int^{\tau_n}_0e^{-\alpha t}f(X^n_t)\,dt
\]
with the convention that $f(X_t)=0$ if $X_t\notin D$. By \cite[Theorem 4.2.4]{FOT},  $R^D_{\alpha}f=G^D_{\alpha}f$ $m$-a.e. 
and $R^{D,n}_{\alpha}f=G^{D,n}_{\alpha}f$ $m$-a.e. for any $\alpha>0$ and bounded Borel $f\in L^2(D)$.

\begin{theorem}
\label{th7.12}
Let $D$ be an open bounded subset of $\BR^d$  of class $C^0$.
If \mbox{\rm(H1)} is satisfied and 
$X$ is a L\'evy process satisfying condition \mbox{\rm(ACP)}, then	
the Dirichlet forms $(\Psi^{(n)}, H^{\psi_n}_0(D))$ converge to  $(\Psi, H^{\psi}_0(D))$ in the sense of Mosco.	
\end{theorem}
\begin{proof}

Fix $\alpha>0$ and let $h$ be a  Borel function in $L^2(D)$.
The functions $R^{D,n}_\alpha h$ and $R^D_\alpha h$ are the
probabilistic solutions corresponding to
$f(x,y)=h(x)-\alpha y$ and $g=0$.
Thus Corollary \ref{cor5.5}, applied with $p=2$, and the
identification of resolvents given above yield
\[
G^{D,n}_\alpha h=R^{D,n}_\alpha h
\longrightarrow R^D_\alpha h=G^D_\alpha h
\quad\text{in }L^2(D).
\]
The conclusion follows from Proposition \ref{prop7.11}.
\end{proof}

In \cite[Theorem 3.3]{CS} the result given in Theorem \ref{th7.12} is stated for any $d\ge1$ and  bounded open subset $D$ of  $\BR^d$. 
The example given below shows that some assumptions on the regularity of  $D$ are indispensable.

\begin{example}
\label{ex7.13}
Let $D=(-1,0)\cup(0,1)$. Let $X$ be a standard Brownian motion and for $n\ge1$ let $X^n$ be a symmetric L\'evy process in $\BR$ with generator
\[
L^n \varphi(x)=\frac{n^2}{2}\Big(\varphi\Big(x+\frac1n\Big)-2\varphi(x)
+\varphi\Big(x-\frac1n\Big)\Big),
\]
i.e. $X^n$ is a compound Poisson process with L\'evy measure $(n^2/2)(\delta_{1/n}+\delta_{-1/n})$ (see Example \ref{ex5.7}). 
An elementary computation shows that $\psi_n(\xi)=n^2(1-\cos(\xi/n))\rightarrow|\xi|^2/2$ as $n\rightarrow\infty$. The corresponding  Dirichlet forms on $L^2(\BR)$ are 
\[
\EE(u,u)=\frac12\int_{\BR}|u'(x)|^2\,dx,
\quad \FF=H^1(\BR),
\]
and
\[
\EE^{(n)}(u,u)=
\frac{n^2}{2}\int_{\BR}\Big(u\Big(x+\frac1n\Big)-u(x)\Big)^2\,dx,
\quad
\FF^{(n)}=L^2(\BR).
\]
The domains of parts of these forms on  $D$ are of the form  
\[
\FF^D=H_0^1(D)=H_0^1(-1,0)\oplus H_0^1(0,1),\qquad \FF^{(n),D}=L^2(D).
\]	
Define
\[
v(x)=(1-|x|)^+,\qquad v_n=v|_D \in L^2(D).
\]
Then $v_n\to v|_D$ strongly in $L^2(D)$. However, $v|_D\notin H_0^1(D)$.
Hence $\EE^D(v,v)=\infty$. On the other hand,
\[
\Big| v\Big(x+\frac1n\Big)- v(x)\Big|\le \frac1n
\quad\text{for }x\in\BR,\qquad  v(x)= v\Big(x+\frac1n\Big),\quad |x|\ge 1+1/n.
\]
Consequently,
\[
\EE^{(n),D}(v_n,v_n)=\frac{n^2}{2}\int_{\BR}
\Big( v\Big(x+\frac1n\Big)- v(x)\Big)^2\,dx
\le
\frac{n^2}{2}\cdot \frac1{n^2}\cdot\Big(2+\frac2n\Big)
=1+\frac1n.
\]
In particular, $\sup_{n\ge1}\EE^{(n),D}(v_n,v_n)<\infty$. We thus have 
\[
v_n\rightharpoonup v \quad\text{in }L^2(D),
\qquad
\liminf_{n\to\infty}\EE^{(n),D}(v_n,v_n)<\infty,
\qquad \EE^D(v,v)=\infty,
\]
so the forms $\EE^{(n),D}$ and $\EE^D$ do not satisfy (M1).  
\end{example}

In Theorem \ref{th7.12} the limiting operator $L$ is a symmetric L\'evy operator. Related stability results covering the case  
where $L$ is a symmetric divergence form operator are found in \cite{FGKV}.
 
\section{Appendix}

We  start with a useful general estimate for the first exit time of a general strong Markov process. We then apply it to get estimate of moments of   first exit times of weakly convergent L\'evy processes. 
\begin{proposition}
\label{prop8.1}
Let $((X_t)_{t\ge 0},(\BP_x)_{x\in E})$ be a strong Markov process with state space  $E$, $D$ be a Borel subset of $E$, and $\tau=\inf\{t>0:X_t\notin D\}$.
Assume that there exist $q>0$ and $p\in(0,1)$ such that
\begin{equation}
\label{eq8.1}
\inf_{x\in D} \BP_x(\tau\le q)\ge p.
\end{equation}
Then, for every $n\in\BN$,
\begin{equation}
\label{eq8.2}
\sup_{x\in D}\BP_x(\tau>n q)\le(1-p)^n.
\end{equation}
Consequently, for every $m\in\BN$,
\begin{equation}
\label{eq8.3}
\sup_{x\in D}\BE_x|\tau|^m\le mq^m \sum_{n=0}^\infty(n+1)^{m-1} (1-p)^n<\infty,
\end{equation}
and moreover, $\sup_{x\in D}\BE_x[e^{\theta \tau}]<\infty$ for every
$\theta<-(1/q)\ln(1-p)$.
\end{proposition}
\begin{proof}
By \eqref{eq8.1},
\begin{equation}
\label{eq8.4}
\sup_{x\in D}\BP_x(\tau>q)\le 1-p.
\end{equation}
Fix $n\in\BN$ and $x\in D$. Using the strong Markov property  we get
\[
\BP_x(\tau>(n+1)q)=\BE_x\big(\fch_{\{\tau>q\}}
\BP_x(\tau>(n+1)q|\FF_{q})\big)
=\BE_x\big(\fch_{\{\tau>q\}}\,\BP_{X_{q}}(\tau>n q)\big).
\]
On the event $\{\tau>q\}$ we have $X_{q}\in D$. Hence
\[
\BP_{X_{q}}(\tau>n q) \le\sup_{y\in D}\BP_y(\tau>n q).
\]
Therefore,
\[
\BP_x(\tau>(n+1)q)\le\BP_x(\tau>q)\cdot \sup_{y\in D}\BP_y(\tau>n q).
\]
From this and  \eqref{eq8.4} we obtain
\[
a_{n+1}\le(1-p)\cdot a_n,\quad\mbox{where}\quad
a_n=\sup_{x\in D}\BP_x(\tau>nq).
\]
Since $a_1\le 1-p$ by \eqref{eq8.4}, induction yields $a_n\le(1-p)^n$, i.e.
\eqref{eq8.2} holds true.
To show (\ref{eq8.3}) we first note that  $\BP_x(\tau>t)\le\BP_x(\tau>nq)\le(1-p)^n$ for $t\in[nq,(n+1)q)$. Hence, for every $x\in D$,
\begin{align*}
\BE_x|\tau|^m
= m\int_0^\infty t^{m-1}\BP_x(\tau>t)\,dt 
&\le m\sum_{n=0}^\infty \int_{nq}^{(n+1)q} t^{m-1}(1-p)^n\,dt\\
&\le mq^m \sum_{n=0}^\infty (n+1)^{m-1}(1-p)^n <\infty,
\end{align*}
which implies (\ref{eq8.3}). Finally, for every 
$\theta\ge 0$, we have
\begin{align*}
\BE_xe^{\theta \tau}
\le 1+ \sum_{n=0}^\infty
\BE_x\big(e^{\theta \tau}\fch_{\{nq<\tau\le (n+1)q\}}\big)
&\le1+ \sum_{n=0}^\infty e^{\theta (n+1)q}\,\BP_x(\tau>nq)\\
&\le 1+e^{\theta q} \sum_{n=0}^\infty \big((1-p) e^{\theta q}\big)^n,
\end{align*}
which is finite whenever $(1-p)e^{\theta q}<1$, i.e.
$\theta<-(\ln(1-p))/q$.
\end{proof}

\begin{lemma}
\label{lem4.2}
Assume that $X$ is a nonconstant L\'evy process in $\BR^d$. Let
$R>0$ and  $\tau_R=\inf\{t>0:|X_t|\ge R\}$. Then for every $t>0$,
\[
\inf_{x\in B(0,R)}\BP_x(\tau_R\le t) \ge p,
\quad\mbox{where}\quad p=\BP_0(|X_{t}|\ge 2R).
\]
\end{lemma}
\begin{proof}
Fix $t>0$.  For every  $x$ such that $|x|<R$ we have
\[
\{|X_{t}^0|\ge 2R\}\ \subseteq\ \{|x+X^0_{t}|\ge R\}
\ \subseteq\ \{\tau_R\le t\}.
\]
This proves the lemma.
\end{proof}

\begin{proposition}
\label{prop4.2}
Let $X, X^n$, $n\ge1$, be nonconstant L\'evy processes in $\BR^d$ and $D\subset \mathbb R^d$ be open and bounded.
If $X^n_1\arrowd X_1$, then for every $q\ge 1$, there exists $C_q>0$ such that 
\begin{equation}
\label{eq8.5}
\BE^n_x|\tau_D|^q\le C_q,\quad x\in D,\, n\ge 1.
\end{equation}
\end{proposition}
\begin{proof}
Choose $R, t_0>0$ so that $D\subset B(0,R):=\{x:|x|<R\}$,  $\BP_0(|X_{t_0}|=2R)=0$,
and $p:=\BP_0(|X_{t_0}|>2R)>0$ (see  \cite[Theorem 24.3]{Sa}).
Define $\tau_R$ as in Lemma \ref{lem4.2}  and set $\tau=\tau_D$ and $p_n=\BP^n_0(|X_{t_0}|>2R)$. 
Clearly  $\{\tau_R\le t_0\}\subset\{\tau\le t_0\}$, so by Lemma \ref{lem4.2},
\[
\inf_{x\in D}\BP^n_x(\tau\le t_0)\ge\inf_{x\in B(0,R)}\BP^n_x(\tau_R\le t_0)\ge p_n.
\]
Since $X^n_{t_0}\arrowd X_{t_0}$, we have $p_n\to p$. Thus, by Proposition \ref{prop8.1} and Lemma \ref{lem4.2}
we get \eqref{eq8.5} with a constant $C^0_q>0$ for $n\ge n_0$, where $n_0$ is such that $p_n>p/2$, $n\ge n_0$.  
By  Proposition \ref{prop8.1} and Lemma \ref{lem4.2}, for any $k\in \{1,\dots,n_0-1\}$, there exists $C^k_q$
such that $\BE^k_x|\tau_D|^q\le C^k_q$, $x\in D$. Therefore the asserted inequality holds with $C_q:= \max\{C^0_q,C^1_q,\dots, C^{n_0-1}_q\}$.
\end{proof}

In what follows  $D(\BR_+;\BR^d)$ is the space of $\BR^d$-valued c\`adl\`ag functions on $\BR_+=[0,\infty)$  equipped  
with the Skorokhod $J_1$ topology. For  $x\in D(\BR_+;\BR^d)$ we write  $\Delta  x_t=x_t-x _{t-}$. 

We shall prove a  result on convergence of first exit times and exit positions of c\`adl\`ag functions 
which converge  in the  $J_1$ topology. 
This result may be known, but we have not found a suitable reference. It may also be of independent interest.

Let $D$ be a  bounded open set  in $\BR^d$, $d\ge1$, and let $\bar D=D\cup\partial D$. For $x,x^n\in D(\BR_+;\BR^d)$ we set
\[
t=\inf\{s>0:x_s\notin D\},\qquad t_n=\inf\{s>0:x^n_s\notin D\}.
\]
\begin{proposition}
\label{prop6.1}
Assume that $x^n\rightarrow x$ in $D(\BR_+;\BR^d)$ as $n\rightarrow\infty$, $x_0\in D$, and moreover
\begin{enumerate}[\rm(a)]
\item for every $0<s<t $, if $\Delta x_s\neq0$, then $x_{s-}\notin\partial D$; 
		
\item	if $t<\infty$, then $t=\bar t$, where $\bar t=\inf\{s>0:x_s\notin \bar D\}$.
\end{enumerate}
Then $t_n\rightarrow t$. If, in addition $t<\infty$ and $x_{t-}\in D$ when $\Delta x_t\neq0$, then $x^n_{t_n}\rightarrow x_t$.
\end{proposition}
\begin{proof}
We first prove the first assertion in the case  $t=\infty$. Striving for a contradiction, suppose that $t_{n'}\rightarrow t'<\infty$ 
for some subsequence $\{t_{n'}\}$.  Then $x_{t'}$, $x_{t'-}$ are the only possible limit points of the sequence $\{x^{n'}_{t_{n'}}\}$. 
Since $x^{n'}_{t_{n'}}\in D^c$ for each $n'$, the limit points belong to $D^c$. If $x_{t'}\in D^c$ is the limit point, then clearly $t\leq t'$ which is a contradiction.
Consider now the case where  $x_{t'-}\in D^c$ is the limit point.
Then, by assumption (a), $x_{t'}\in D^c$, which means as before that $t\leq t'$ and leads to a contradiction.
	
Assume now that $t<\infty$.
As in  the first   part of the proof one can show that there is no subsequence $\{t_{n'}\}$ such that $t_{n'}\rightarrow t'<t$. Suppose that there is a 
subsequence  $\{t_{n''}\}$ such that $t_{n''}\rightarrow t''>t$. Then, by assumption (b) and
right-continuity of $x$, there are  $t_1,t_2\in(t,t'')$
such that $x_s\in(\bar D)^c$ for $s\in(t_1,t_2]$. Let $s^*$ be such that
\begin{equation}
\label{eq8.6}
t<t_1<s^*<t_2<t''\quad\mbox{and}\quad \Delta x_{s^*}=0.
\end{equation}
Then $x^{n''}_{s^*}\rightarrow x_{s^*}$ and $x^{n''}_{s^*}\in(\bar D)^c$ for all sufficiently large $n''$. This implies that $s^*\ge t_{n''}$ for a 
sufficiently large $n''$. Hence $s^*\ge t''$, which contradicts (\ref{eq8.6}). 

We now turn to the proof of the second assertion. We already know that $t_n\rightarrow t$, so $x_t,x_{t-}$ are the only 
possible limit points of the sequence $\{x^n_{t_n}\}$. Since $x^n_{t_n}\in D^c$, $n\ge1$, these limit points belong to $D^c$. 
Hence, in fact, $x_t$ is the only limit point.
\end{proof}

\begin{example}
\label{ex6.2}
(i) Let $D=(-1,1)$ and 
\[
x_s=\begin{cases}
s, &s\in[0,1),\\
1/2, &s\in[1,2),\\
2, & s\in[2,\infty),
\end{cases}
\qquad x^n_s=\begin{cases}
s+1/n, &s\in[0,1),\\
1/2, &s\in[1,2),\\
2, & s\in[2,\infty).
\end{cases}
\]
Then $x^n\rightarrow x$ in $D(\BR_+;\BR)$ but condition (a) of Proposition \ref{prop6.1} is not satisfied because 
$t=2$ and for $s=1$ we have $\Delta x_s=-1/2$ and $x_{s-}=1\in\partial D$. Clearly, we have $t_n=1-1/n\rightarrow1<2=t$.\smallskip\\
(ii)  Let $D=(-1,1)$ and 
\[
x_s=\begin{cases}
0, &s\in[0,1),\\
1, &s\in[1,2),\\
2, & s\in[2,\infty),
\end{cases}
\qquad x^n_s=\begin{cases}
0, &s\in[0,1),\\
1-1/(2n), &s\in[1,2),\\
2, & s\in[2,\infty).
\end{cases}
\]
Then $x^n\rightarrow x$ in $D(\BR_+;\BR)$ but condition (b)  is not satisfied because $t=1$ and $\bar t=2$. Clearly, $t_n=2\nrightarrow t=1$.
\smallskip\\
(iii) Let $D=(-1,1)$ and
\[
x_s=\begin{cases}
s, &s\in[0,1),\\
2, &s\in[1,\infty),
\end{cases}
\qquad 
x^n_s=\begin{cases}
s+1/n, & s\in[0,1),\\
2, &s\in[1,\infty).
\end{cases}
\]
Then $x^n\rightarrow x$, conditions (a), (b) are satisfied  and $t_n=1-1/n\rightarrow1=t$. However, $x^n_{t_n}=1\nrightarrow x_t=2$.
\end{example}

\section*{Statements and Declarations}

\noindent\textbf{Competing interests.}
The authors have no relevant financial or non-financial interests to disclose.

\section*{Data Availability Statement}

No datasets were generated or analysed during the current study.

\section*{Acknowledgments}
This work was supported by the Polish National Science Centre\\
(Grant No.~2022/45/B/ST1/01095).

\end{document}